\documentclass{siamart220329}

\usepackage{siampaper}

\usepackage{multirow}
\newcommand{\emphdef}[1]{\emph{#1}}
\usepackage{enumitem}

\newcommand{\cahiernumber}{42}  

\usepackage[hyperpageref]{backref}
\renewcommand*{\backref}[1]{}
\renewcommand*{\backrefalt}[4]{%
\ifcase #1 %
No citations.%
\or
\marginpar{\tiny cited on p. #2}%
\else
\marginpar{\tiny cited on pp. #2}%
\fi
}
\crefname{subsection}{section}{sections}
\Crefname{subsection}{Section}{Sections}

\def\theTitle{%
  Envelopt:~Constrained~Convex~Composite~Optimization%
}

\def\theKeywords{%
  Convex composite constrained optimization,
  nonlinear programming,
  augmented Lagrangian method,
  proximal algorithms,
  Moreau envelope%
}

\def\authorADM{Alberto De~Marchi}
\def\emailADM{alberto.demarchi@unibw.de}
\def\orcidADM{0000-0002-3545-6898}
\def\affiliationADM{%
	Institute of Applied Mathematics and Scientific Computing,
	Department of Aerospace Engineering,
	University of the Bundeswehr Munich,
	Germany%
}

\def\authorDPO{Dominique Orban}
\def\emailDPO{dominique.orban@gerad.ca}
\def\orcidDPO{0000-0002-8017-7687}
\def\affiliationDPO{%
	GERAD and Department of Mathematics and Industrial Engineering,
	Polytechnique Montr{\'e}al,
	Canada%
}
\newcommand\nsercACK[1]{Research partially supported by NSERC Discovery Grant RGPIN–#1}

\hypersetup{
  pdftitle={\theTitle},
  pdfauthor={\authorADM{} and \authorDPO{}},
  pdfsubject={Convex composite constrained optimization},
  pdfkeywords={\theKeywords},
}

\makeatletter
\newcommand*{\coloneqq}{\mathrel{\vcenter{\baselineskip0.5ex \lineskiplimit0pt
                     \hbox{\scriptsize.}\hbox{\scriptsize.}}}%
                     =}
\makeatother
\newcommand{\orcidLink}[1]{\href{https://orcid.org/#1}{#1}}
\newcommand{\amsmscLink}[1]{\href{http://www.ams.org/mathscinet/msc/msc2020.html?t=#1}{#1}}
\DeclareMathOperator{\dom}{dom}
\DeclareMathOperator{\dist}{dist}
\newcommand{\lagr}{\mathcal{L}}
\newcommand{\elim}{\mathop{\textup{e-lim}}}

\newcommand{\identity}{\mathbb{I}}
\newcommand{\ceil}[1]{\left\lceil #1 \right\rceil}

\title{\theTitle}

\author{
  \MakeUppercase\authorADM\thanks{%
    \affiliationADM.
    E-mail: \mailto{\emailADM}.
    ORCID\@: \orcidLink{\orcidADM}.
  }
  \and
  \authorDPO\thanks{%
    \affiliationDPO.
    E-mail: \mailto{\emailDPO}.
    ORCID\@: \orcidLink{\orcidDPO}.
    \funding{\nsercACK{2020–06535}}.
  }
}
\date{\today}
\newcommand{\TheAbstract}{%
	We introduce Envelopt, a globally convergent iterative framework for a broad class of structured optimization problems where a smooth objective is augmented by a nonsmooth convex regularizer composed with a smooth mapping, and the variables are subject to general smooth constraints.
	All smooth functions may be nonconvex.
	The method is akin to an augmented-Lagrangian method in which partial minimization with respect to a lifting variable results in smooth subproblems involving the Moreau envelope of the nonsmooth regularizer, and the original constraints are retained explicitly.
	Only the proximal operator of the regularizer is required.
	Subproblems may be solved with off-the-shelf smooth optimization solvers.
	We state global convergence properties, establish that feasible limit points are asymptotically stationary, and develop an infeasibility detection mechanism.
	We derive worst-case iteration complexity bounds when the penalty parameter is and is not bounded away from zero.
	The framework subsumes the classical augmented Lagrangian method and accommodates important extensions, including stabilized formulations for degenerate problems, exact penalty methods, and conic constraints.
	We provide a Julia implementation, \texttt{Envelopt.jl}, as part of the JuliaSmoothOptimizers ecosystem.
	Numerical experiments with low-rank matrix completion, semidefinite programming, complementarity-constrained optimization, and nonconvex regularizers demonstrate the effectiveness and versatility of Envelopt.%
}
\newcommand{\TheMSCClasses}{%
	\amsmscLink{49M37}, 
	\amsmscLink{49J53},  
	\amsmscLink{65K05},  
	\amsmscLink{65K10},  
	\amsmscLink{90C06}
}

\begin{document}

\maketitle

\thispagestyle{firstpage}
\pagestyle{myheadings}

\begin{abstract}
	\TheAbstract
\end{abstract}

\begin{keywords}
	\theKeywords{}.
\end{keywords}

\begin{AMS}
	\TheMSCClasses
\end{AMS}

\section{Introduction}%
\label{sec:introduction}

We consider the problem
\begin{equation}%
	\label{eq:constrained-convex-composite}
	\minimize{x \in \R^n}\quad
	f(x) + h(F(x))
	\quad \st \quad
	c(x) \in C,
\end{equation}
where \(f: \R^n \to \R\), \(F: \R^n \to \R^p\), and \(c: \R^n \to \R^m\) are \(C^1\), \(h: \R^p \to \R \cup \{+\infty\}\) is proper, lower semicontinuous (lsc), and convex, and \(C \subseteq \R^m\) is nonempty, closed, and convex.
Even though \(f\), \(F\) and \(c\) may be nonconvex, we call~\eqref{eq:constrained-convex-composite} a \emph{constrained convex composite problem} because of the crucial role played by the convexity of $h$ \citep{burke1985descent}.

Our approach consists in solving~\eqref{eq:constrained-convex-composite} by solving a sequence of smooth problems
\begin{equation}%
	\label{eq:partially-eliminated-subproblem}
	\minimize{x \in \R^n}\quad
	f(x) + h_\mu(F(x) + \mu \hat{y})
	\quad \st \quad
	c(x) \in C,
\end{equation}
where \(h_\mu: \R^n \to \R\) is the \(C^1\) Moreau envelope of \(h\), \(\mu > 0\), \(\hat{y} \in \R^p\), and \(c\) is as in~\eqref{eq:constrained-convex-composite}.
Our main working assumptions are that an efficient solver for such smooth problems is available, and that the proximal operator of \(h\) is available.
Some or all components of \(F(x)\) could originate from constraints that were moved to the objective via an indicator.
In our approach, such constraints will, in general, not be satisfied along the iterations.
In that sense, constraints in \(c\) may be viewed as \emph{hard} or \emph{unrelaxable}.
If any constraint is \emph{soft}, or \emph{relaxable}, the user has the option of encoding it into \(h\) by way of an indicator provided that the proximal operator remains available.
Keeping a constraint explicit means our solver can, at the user's option, enforce it throughout, but requires a subproblem solver for~\eqref{eq:partially-eliminated-subproblem} capable of handling that constraint type; encoding it into $h$ removes that requirement but forces the solver to discover feasibility through penalization.

Numerical experiments on low-rank matrix completion, semidefinite programming, and complementarity-constrained optimization confirm that Envelopt is competitive with or outperforms state-of-the-art solvers, including on degenerate problems where standard methods fail.

We provide a Julia implementation of Envelopt with examples named \texttt{Envelopt.jl} that is available from \https{github.com/JuliaSmoothOptimizers/Envelopt.jl} as part of the JSO ecosystem.

\subsection*{Contributions}

Our main contributions are as follows:
\begin{enumerate}
    \item we provide a globally-convergent general framework for constrained convex composite optimization based on the Moreau envelope that entirely relies on smooth subproblems---see \Cref{sec:methodology};
    \item our convergence analysis is based on asymptotic satisfaction of stationarity conditions and does not assume any constraint qualification--see \Cref{sec:analysis};
    \item we extract the features of a concrete algorithm that ensure convergence and discuss whether classic augmented-Lagrangian methods have those features--see \Cref{alg:envelopt} and \Cref{asm:envelopt-multipliers};
    \item we provide a worst-case complexity analysis whether the penalty parameter remains bounded away from zero or not---see \Cref{thm:complexity-mu-bounded} and \Cref{thm:complexity-mu-zero};
    \item we review important extensions and applications of our framework, particularly NCL, exact penalty methods, and the proximal operator of composite sums---see \Cref{sec:extensions-applications};
    \item we provide a complete Julia implementation of an instance of an algorithm that has the required features, named Envelopt---see \Cref{sec:concrete_implementations_envelopt,sec:implementation-details};
\end{enumerate}

\subsection*{Related work}

The augmented Lagrangian method has been a powerful optimization tool for many decades, starting with \citet{hestenes-1969,powell-1969} and \citet{rockafellar1976augmented} for smooth problems.
Several efficient numerical schemes have been proposed, including those of \citet{conn1991globally} and \citet{birgin2014practical}.

In the context of nonlinear programming, \citet{andreani2008augmented} investigated augmented Lagrangian methods in which the lower-level constraints of subproblems need not be simple bounds, proving convergence to approximate KKT points under weak constraint qualifications.
In a related spirit, \citet{birgin2016sequential} proposed SECO, in which subproblems are equality-constrained smooth NLPs while bound constraints are handled by an outer augmented Lagrangian loop.
Envelopt shares this philosophy of splitting constraints between those kept explicit in subproblems and those penalized in outer iterations.

Extensions to nonsmooth problems include those of \citet{rockafellar2022convergence} with \(h\) convex, and \citet{hallak2023adaptive,demarchi2023constrained,demarchi2025penalty} with \(h\) nonconvex.
The latter works hinge around unconstrained or proximal subproblems.
The combination of nonconvex $h$ with general hard constraints $c(x) \in C$ remains an open problem.

\citet{dhingra2019proximal} focus on~\eqref{eq:constrained-convex-composite} with convex $h$, without explicit constraints, and assume that $F$ is a bounded linear operator.
However, they already observed that, by constraining the augmented Lagrangian to the manifold that corresponds to the explicit minimization over the auxiliary variable, the subproblem involves the Moreau envelope of the nonsmooth regularizer and is continuously differentiable.
\citet{dhingra-khong-jovanovic-2022} augment their approach to use a generalized Hessian for the Moreau envelope.
\citet{rockafellar2022convergence} refers to the objective of~\eqref{eq:partially-eliminated-subproblem} as the \emph{generalized} augmented Lagrangian.
These works lay the groundwork for the smooth-subproblem approach that Envelopt builds upon, while being limited to the unconstrained setting $C=\R^m$.

\citet{demarchi2024implicit} addresses~\eqref{eq:constrained-convex-composite} with \(h\) possibly nonconvex, but without hard constraints $c(x)\in C$, using the same partial-minimization approach over the lifting variable.
Envelopt removes this restriction by retaining $c(x) \in C$ explicitly in every subproblem, which can be accounted for by any off-the-shelf constrained solver.

A different line of work avoids the augmented Lagrangian altogether.
The \emph{prox-linear} scheme of \citet{lewis2016proximal} addresses the minimization of functions that are composition of a prox-regular $h$ with a smooth $F$.
Their proximal linearized subproblem involves the composition of $h$ with a linearization of $F$, which can be easily evaluated for well-structured problems only.

When $F=\identity$, \citet{chouzenoux2020proximal} introduce a proximal interior point algorithm able to handle problems with convex $f$, $h$, and inequality constraints.
\citet{leconte-orban-2024} develop an interior-point method for bound-constrained problems in which both $f$ and $h$ are allowed to be nonconvex,
and \citet{demarchi2024interior} further consider nonconvex inequality constraints.

\citet{dai2025proximal} investigate a numerical scheme for the special case where $F = \identity$, $h$ real- and nonnegative-valued, and only equality constraints are present.
\citet{curtis2025proximal} cover nonlinear inequality constraints, but their decomposition procedure accesses the regularizer by solving a constrained subproblem, which is tractable in practice only for special cases such as $h \coloneqq \|\cdot\|_1$.
Envelopt sidesteps these limitations by accessing $h$ only through its proximal operator, inheriting the full generality of the augmented Lagrangian framework while retaining hard constraints explicitly.

\subsection*{Notation}

We use lowercase Latin letters for vectors in \(\R^n\) and lowercase Greek letters for scalars.
Exceptions are that \(f\), \(g\) and \(h\) are functions.
Uppercase Latin letters are sets, except \(F\) and \(L\), which are functions.
For \(x \in \R^n\), we let \(\nabla F(x)\) and \(\nabla c(x)\) be the \emph{gradients} of \(F\) and \(c\) at \(x\), i.e., the \(n \times p\) and \(n \times m\) matrices whose \(i\)-th column is \(\nabla F_i(x)\) and \(\nabla c_i(x)\), respectively, where \(F(x) = (F_1(x), \ldots, F_p(x))\) and \(c(x) = (c_1(x), \ldots, c_m(x))\).
The identity function is denoted \(\identity\).

\section{Background}

The domain of \(h\) is \(\dom h \coloneqq \{x \in \R^n \mid h(x) < +\infty\}\).
\(h\) is proper if \(\dom h \neq \varnothing\) and lower semicontinuous if \(h(x) \leq \liminf_{y \to x} h(y)\) for all \(x \in \R^n\).

If \(C \subseteq \R^n\), the indicator of \(C\) is \(\chi(x \mid C) \coloneqq 0\) if \(x \in C\) and \(+\infty\) otherwise.
When \(C\) is convex, \(\chi(\cdot \mid C)\) is convex as well, and, when \(C\) is closed, \(\chi(\cdot \mid C)\) is lsc.
If \(C\) is closed, the Euclidean distance function to \(C\) is \(\dist(x \mid C) \coloneqq \min \{\|x - y\|_2 \mid y \in C\}\).

For \(f_1\), \(f_2: \R^n \to \R \cup \{\pm \infty\}\), the infimal convolution of \(f_1\) and \(f_2\) is the function denoted \(f_1 \square f_2\) defined by
\[
	(f_1 \square f_2)(x) \coloneqq \inf \{f_1(y) + f_2(x - y) \mid y \in \R^n\}.
\]
A special case is \(\dist(\cdot \mid C) = \chi(\cdot \mid C) \square \|\cdot\|_2\).

The Fenchel conjugate of \(h\) is \(h^*: \R^n \to \R \cup \{+\infty\}\) defined by
\[
	h^*(y) \coloneqq \sup \{y^\top x - h(x) \mid x \in \R^n\}.
\]
The only function that coincides with its Fenchel conjugate is \(\tfrac{1}{2} \|\cdot\|_2^2\).

The epigraph of \(h\) is the set \(\{(x, \alpha) \mid h(x) \leq \alpha\} \subseteq \R^n \times (\R \cup \{\pm \infty\})\).
If \(\{h_k\}\) is a sequence of functions, we say that \(h_k\) epi-converges to \(h\) if the epigraphs of \(h_k\) converge to the epigraph of \(h\) in the Painlevé--Kuratowski sense.

For \(u \in \R^n\) and \(\mu > 0\), the \emphdef{proximal operator} of \(h\) is
\begin{equation}%
	\label{eq:def-prox}
	\prox{\mu h}(u) \coloneqq \argmin{v \in \R^n} \left\{ \tfrac{1}{2} \|v - u\|_2^2 + \mu h(v) \right\}.
\end{equation}
The \emphdef{Moreau envelope} of \(h\) with parameter \(\mu\), introduced by \citet{moreau1965proximite} and denoted \(h_\mu\), is
\begin{equation}%
	\label{eq:def-moreau-envelope}
    h_\mu(u) \coloneqq \inf_{v \in \R^n} \left\{ \tfrac{1}{2} \mu^{-1} \|v - u\|_2^2 + h(v) \right\} = \left(h \square (\tfrac{1}{2} \mu^{-1} \|\cdot\|_2^2\right))(u),
\end{equation}
and is the optimal value function in~\eqref{eq:def-prox}.

Because \(h\) is convex,~\eqref{eq:def-prox} and~\eqref{eq:def-moreau-envelope} enjoy strong properties.

\begin{proposition}[{\protect \citealp[Theorem~\(2.26\)]{rockafellar1998variational}}]%
	\label{prop:prox-moreau}
	Assume that \(h: \R^n \to \R \cup \{+\infty\}\) is proper, lsc and convex.
	For every \(\mu > 0\), \(\prox{\mu h}\) is single valued and continuous, and \(h_\mu\) is convex and \(C^1\) with
	\begin{equation}%
		\label{eq:Moreau_envelope_deriv}
		\nabla h_\mu(u) = \frac{ u - \prox{\mu h}(u) }{ \mu }.
	\end{equation}
\end{proposition}

Though~\eqref{eq:def-prox} is a set, \Cref{prop:prox-moreau} indicates that it is single-valued under our assumptions, and we will abuse notation and write \(v = \prox{\mu h}(u)\) to mean \(\{v\} = \prox{\mu h}(u)\).
In general, \(h_\mu\) is not twice differentiable, but its gradient is Lipschitz continuous with constant \(\mu^{-1}\) \citep[Exercise~\(12.23\)]{rockafellar1998variational}.

In addition, there is no need to appeal to the general concept of limiting subdifferential of \(h\) \citep[Definition~\(8.3\)b]{rockafellar1998variational} for the latter coincides with the Fréchet subdifferential \citep[Definition~\(8.3\)a]{rockafellar1998variational}, and both coincide with the subdifferential of convex analysis \citep[Proposition~\(8.12\)]{rockafellar1998variational}, i.e.,
\begin{equation}%
	\label{eq:def-subdifferential}
	\partial h(x) \coloneqq \{ v \in \R^n \mid h(y) \geq h(x) + v^\top (y - x) \text{ for all } y \in \R^n \}
	\quad
	(x \in \dom h).
\end{equation}

For \(\bar{y} \in C\), the normal cone to \(C\) at \(\bar{y}\) is
\begin{equation}%
	\label{eq:def-normal-cone}
	N_C(\bar{y}) \coloneqq \{v \in \R^m \mid v^\top (y - \bar{y}) \leq 0 \text{ for all } y \in C \}.
\end{equation}

We say that \(x \in \R^n\) is \emphdef{feasible} for~\eqref{eq:constrained-convex-composite} if \(F(x) \in \dom h\) and \(c(x) \in C\).
Such feasible \(x\) is \emphdef{stationary} for~\eqref{eq:constrained-convex-composite} if there exist \(y \in \R^p\) and \(z \in \R^m\) such that
\begin{subequations}%
	\label{eq:KKT}
	\begin{align}
		0 ={}   & \nabla f(x) + \nabla F(x) y + \nabla c(x) z \\
		y \in{} & \partial h(F(x))                            \\
		z \in{} & N_C(c(x)).
	\end{align}
\end{subequations}

A local minimizer of~\eqref{eq:constrained-convex-composite} is only guaranteed to be stationary under a constraint qualification.
However, an asymptotic concept of stationarity holds without constraint qualification.

Let \(x \in \R^n\) be feasible for~\eqref{eq:constrained-convex-composite}.
We say that \(x\) is \emphdef{asymptotically stationary} if there exist sequences $\{x^k\} \to x$, $\{u^k\} \subseteq \dom h$, $\{y^k\} \subset \R^n$, $\{w^k\} \subset \R^m$ and $\{z^k\} \subset \R^m$ with \(y^k \in \partial h(u^k)\) and \(z^k \in N_C(w^k)\) for all \(k \in \N\) such that
\begin{equation}
	\label{eq:def-AKKT}
	\nabla f(x^k) + \nabla F(x^k) y^k + \nabla c(x^k) z^k \to 0,
    \quad
	F(x^k) - u^k \to 0,
    \quad\text{and}\quad 
	c(x^k) - w^k \to 0.
\end{equation}

In~\eqref{eq:def-AKKT}, the slack variables \(u^k\) and \(w^k\) allow for the constraints $F(x)\in\dom h$ and $c(x)\in C$ to be violated along the iterations, yet satisfied in the limit.
The following result connects local optimality and asymptotic stationarity.

\begin{proposition}[{\protect \citealp[Proposition~\(2.5\)]{demarchi2023constrained}}]
	If $\bar{x} \in \R^n$ is a local minimizer of~\eqref{eq:constrained-convex-composite}, it is asymptotically stationary.
\end{proposition}

Let \(\epsilon > 0\).
For the purpose of designing a stopping condition in an iterative scheme, we say that \(x \in \R^n\) is \emphdef{\(\epsilon\)-stationary} if there exist $u \in \dom h$, $w \in C$, $y \in \partial h(u)$ and $z \in N_C(w)$ such that
\begin{equation}%
	\label{eq:approxKKT}
	\| \nabla f(x) + \nabla F(x) y + \nabla c(x) z \| \leq \epsilon,
	\quad
	\| F(x) - u \| \leq \epsilon
	\quad\text{and}\quad
	\| c(x) - w \| \leq \epsilon.
\end{equation}

\section{Methodology}%
\label{sec:methodology}

At any \(x \in \R^n\), we assume that it is possible to evaluate \(f\) and \(\nabla f\), the proximal operator of \(h\) (with the value attained at the proximal point), and \(F\) and gradient-vector products \(\nabla F(x) v\) for \(v \in \R^p\).

A major obstacle if one wishes to apply a proximal method to~\eqref{eq:constrained-convex-composite} is that the proximal operator of \(h \circ F\) is not available in general, even if that of \(h\) is.

We begin by introducing a lifting variable \(u \in \R^p\) and reformulating~\eqref{eq:constrained-convex-composite} as
\begin{equation}%
	\label{eq:lifted-problem}
	\minimize{x \in \R^n, u \in \R^p}\quad
	f(x) + h(u)
	\quad \st \quad
	F(x) - u = 0, \quad c(x) \in C.
\end{equation}
As may be expected,~\eqref{eq:lifted-problem} is related to~\eqref{eq:constrained-convex-composite} in the following way.

\begin{lemma}[{\protect \citealp[Lemma~\(3.1\)]{demarchi2023constrained}}]
	A feasible point $\bar{x} \in \R^n$ of~\eqref{eq:constrained-convex-composite} is asymptotically stationary for~\eqref{eq:constrained-convex-composite} if and only if $(\bar{x}, F(\bar{x}))$ is asymptotically stationary for~\eqref{eq:lifted-problem}.
\end{lemma}

At this point, a nonsmooth augmented-Lagrangian method such as ALPS \citep{demarchi2023constrained} could be applied to~\eqref{eq:lifted-problem}.
However,~\eqref{eq:lifted-problem} has \(p\) more variables and constraints than~\eqref{eq:constrained-convex-composite}.
Fortunately, the augmented-Lagrangian subproblem associated with~\eqref{eq:lifted-problem} is amenable to partial minimization with respect to \(u\), and that results in a smooth problem in \(x\) alone involving the Moreau envelope of \(h\).
We now detail the procedure.

Let \(\mu > 0\) be a penalty parameter and \(\hat{y} \in \R^p\) be a vector of Lagrange multiplier estimates for the equality constraint \(F(x) - u = 0\).
The (partial) augmented Lagrangian function for~\eqref{eq:lifted-problem} is
\begin{equation}%
	\label{eq:augmented-lagrangian}
	\begin{aligned}
		\lagr(x, u; \hat{y}, \mu) \coloneqq{} & f(x) + h(u) + \hat{y}^\top (F(x) - u) + \tfrac{1}{2} \mu^{-1} \|F(x) - u\|_2^2
		\\
		={}                                   & f(x) + h(u) + \tfrac{1}{2} \mu^{-1} \|F(x) - u + \mu \hat{y}\|_2^2 - \tfrac{1}{2} \mu \|\hat{y}\|_2^2.
	\end{aligned}
\end{equation}
For fixed \(\hat{y}\) and \(\mu > 0\), the augmented-Lagrangian subproblem is
\begin{equation}%
	\label{eq:augmented-lagrangian-subproblem}
	\minimize{x \in \R^n, u \in \R^p}\quad
	\lagr(x, u; \hat{y}, \mu)
	\quad
	\st \quad c(x) \in C.
\end{equation}
In~\eqref{eq:augmented-lagrangian-subproblem}, we kept the constraints \(c(x) \in C\) explicit, while we relaxed \(F(x) = u\) and penalized its violation.
Moreover, \(u\) is unconstrained, and the objective is strongly convex in \(u\).
Hence, partially minimizing~\eqref{eq:augmented-lagrangian} with respect to \(u\) is appealing.
By~\eqref{eq:def-moreau-envelope}, the optimal value of~\eqref{eq:augmented-lagrangian} with respect to \(u\) is
\begin{equation}%
	\label{eq:def-partially-eliminated-AL}
	\begin{aligned}
		L(x; \hat{y}, \mu) \coloneqq{} &
		\min_{u \in \R^p} \lagr(x, u; \hat{y}, \mu) = f(x) + h_\mu(F(x) + \mu \hat{y}) - \tfrac{1}{2} \mu \|\hat{y}\|_2^2
		\\
		={}                            & \lagr(x, u(x; \hat{y}, \mu); \hat{y}, \mu),
	\end{aligned}
\end{equation}
where \(u(x; \hat{y}, \mu) = \prox{\mu h}(F(x) + \mu \hat{y}) \in \dom h\) realizes the minimum.
Thus, we replace~\eqref{eq:augmented-lagrangian-subproblem} with the partially-eliminated augmented-Lagrangian subproblem~\eqref{eq:partially-eliminated-subproblem}, which has a decision space of dimension $n$ instead of $n+p$, and whose objective and constraints are \(C^1\) by \Cref{prop:prox-moreau} and our assumptions on \(f\), \(F\) and \(c\).

For reference, \Cref{prop:prox-moreau} implies that the gradient of the objective in~\eqref{eq:partially-eliminated-subproblem} is
\begin{equation*}
	\nabla f(x) + \nabla F(x) \nabla h_\mu(F(x) + \mu \hat{y}) =
	\nabla f(x) + \nabla F(x) y(x; \hat{y}, \mu) .
\end{equation*}
where
\begin{equation}
	y(x; \hat{y}, \mu)
	\coloneqq
	\hat{y} + \frac{ F(x) - u(x; \hat{y}, \mu) }{ \mu }
	.
\end{equation}
Thus, \(x \in \R^n\) is stationary for~\eqref{eq:partially-eliminated-subproblem} if there exists \(z \in \R^m\) such that
\begin{equation}%
	\label{eq:KKT_subproblem}
	\nabla f(x) + \nabla F(x) y(x; \hat{y}, \mu) + \nabla c(x) z = 0,
    \quad\text{and}\quad
	z \in N_C(c(x)).
\end{equation}

For \(\epsilon > 0\), we say that \(x \in \R^n\) is \(\epsilon\)-stationary for~\eqref{eq:partially-eliminated-subproblem} if there exist \(w \in C\) and \(z \in N_C(w)\) such that
\begin{equation}%
	\label{eq:approxKKT_subproblem}
	\| \nabla f(x) + \nabla F(x) y(x; \hat{y}, \mu) + \nabla c(x) z \| \leq \epsilon,
    \quad\text{and}\quad
	\| c(x) - w \| \leq \epsilon.
\end{equation}

\begin{lemma}%
	\label{lem:eps-stationarity}
	Let \(\mu > 0\), \(\hat{y}\in\R^p\), \(\epsilon > 0\), and \(x \in \R^n\) be \(\epsilon\)-stationary for~\eqref{eq:partially-eliminated-subproblem}.
	If \(\|F(x) - u(x; \hat{y}, \mu)\| \leq \epsilon\),
	then \(x\) is \(\epsilon\)-stationary for~\eqref{eq:constrained-convex-composite}.
\end{lemma}

\begin{proof}
	By assumption, there exist \(w \in C\) and \(z \in N_C(w)\) such that~\eqref{eq:approxKKT_subproblem} holds.
	By definition of \(u(x; \hat{y}, \mu)\), \(y(x; \hat{y}, \mu)\) and~\eqref{eq:def-prox},
	\(y(x; \hat{y}, \mu) \in \partial h(u(x; \hat{y}, \mu))\) holds, and so does~\eqref{eq:approxKKT}.
\end{proof}

\section{Algorithm and convergence analysis}%
\label{sec:analysis}

We state a generic template algorithm to solve~\eqref{eq:constrained-convex-composite} by solving a sequence of problems of the form~\eqref{eq:partially-eliminated-subproblem} as \Cref{alg:envelopt}.
Several concrete algorithms fit the template and will be detailed in \Cref{sec:concrete_implementations_envelopt}.
For the purposes of the convergence analysis, we state the features that such a concrete algorithm should possess.

\begin{algorithm}[ht]%
	\caption[caption]{%
		\label{alg:envelopt}
		Envelopt: template algorithm 
	}
	\begin{algorithmic}[1]%
		\State Choose a stopping tolerance \(\epsilon > 0\).
		\For{\(k = 0, 1, \ldots\)}
		\State%
		\label{step:select-mu-y}%
		Select suitable \(\epsilon_k > 0\), \(\mu_k > 0\), and \(\hat{y}_k \in \R^p\).
		\State Find an \(\epsilon_k\)-stationary point \(x_k\) of~\eqref{eq:partially-eliminated-subproblem}.
		\State Compute \(u_k = u(x_k; \hat{y}_k, \mu_k)\) and \(y_k = y(x_k; \hat{y}_k, \mu_k)\).
		\State%
		\label{stp:eps-kkt-check}%
		If \(x_k\) is \(\epsilon\)-stationary for~\eqref{eq:constrained-convex-composite}, return \(x_k\).
		\EndFor
	\end{algorithmic}
\end{algorithm}

Let \(x_k\) and \(u_k\) be generated by \Cref{alg:envelopt} at iteration \(k\).
By \Cref{lem:eps-stationarity}, if \(\epsilon_k \leq \epsilon\) and \(\|F(x_k) - u_k\| \leq \epsilon\), \(x_k\) is an \(\epsilon\)-stationary point of~\eqref{eq:constrained-convex-composite}, and \Cref{alg:envelopt} terminates at step~\ref{stp:eps-kkt-check}.

\subsection{Convergence analysis}

We make the following assumption that a concrete implementation of \Cref{alg:envelopt} should satisfy.

\begin{assumption}%
	\label{asm:envelopt-multipliers}
	The sequences \(\{\mu_k\}\), \(\{x_k\}\), \(\{u_k\}\) and \(\{\hat{y}_k\}\) generated by \Cref{alg:envelopt} are such that
	\begin{enumerate}[label=(\roman{*})]
		\item if there exists \(\mu > 0\) such that \(\mu_k \geq \mu\) for all \(k\), then, \(\{F(x_k) - u_k\} \to 0\);
		\item if there is an index set \(K \subseteq \N\) such that \(\{\mu_k\}_K \to 0\), then \(\{\mu_k \hat{y}_k\}_K \to 0\).
	\end{enumerate}
\end{assumption}

The two conditions in \Cref{asm:envelopt-multipliers} are related to mutually exclusive scenarios:
if the penalty parameter remains bounded away from zero, then the iterates approach feasibility, otherwise the Lagrange multiplier estimates ``do not behave too badly,'' as suggested by \citet[\S 3]{conn1991globally}.

\begin{theorem}%
	\label{thm:convergence_AKKT}
	Let \(\{x_k\}\) be a sequence generated by \Cref{alg:envelopt} with \(\{\epsilon_k\} \to 0\) and let \Cref{asm:envelopt-multipliers} be satisfied.
	Then, any limit point of \(\{x_k\}\) that is feasible for~\eqref{eq:constrained-convex-composite} is asymptotically stationary for~\eqref{eq:constrained-convex-composite}.
\end{theorem}

\begin{proof}
	By construction, each \(x_k\) is \(\epsilon_k\)-stationary for~\eqref{eq:partially-eliminated-subproblem}.
	Let \(\bar{x}\) be a limit point of \(\{x_k\}\) that is feasible for~\eqref{eq:constrained-convex-composite}, and let \(K \subseteq \N\) be an index set such that \(\{x_k\}_K \to \bar{x}\).
	Because \(\{\epsilon_k\}_K \to 0\), \Cref{lem:eps-stationarity} implies that, if \(\{F(x_k) - u_k\}_K \to 0\), \(\bar{x}\) is asymptotically stationary for~\eqref{eq:constrained-convex-composite}.

	There are two cases.
	In the first case, there exists \(\mu > 0\) such that \(\mu_k \geq \mu\) for all \(k\).
	Clearly, \Cref{asm:envelopt-multipliers} then implies that \(\{F(x_k) - u_k\}_K \to 0\).

	In the second case, \(\{\mu_k\}_K \to 0\).
	By definition of \(u_k\), for all \(u \in \R^p\) and all \(k\),
	\[
		\tfrac{1}{2} \|F(x_k) + \mu_k \hat{y}_k - u_k\|_2^2 + \mu_k h(u_k) \leq \tfrac{1}{2} \|F(x_k) + \mu_k \hat{y}_k - u\|_2^2 + \mu_k h(u).
	\]
	In particular, for \(u = F(\bar{x})\in\dom h\),
	\[
		\tfrac{1}{2} \|F(x_k) + \mu_k \hat{y}_k - u_k\|_2^2 + \mu_k h(u_k) \leq \tfrac{1}{2} \|F(x_k) + \mu_k \hat{y}_k - F(\bar{x})\|_2^2 + \mu_k h(F(\bar{x})).
	\]
	We take the limit inferior over \(k \in K\) on both sides and use \Cref{asm:envelopt-multipliers} to obtain
	\begin{equation}%
		\label{eq:limit-inequality}
		\liminf_{k \in K} \tfrac{1}{2} \|F(\bar{x}) - u_k\|_2^2 + \mu_k h(u_k) \leq 0.
	\end{equation}
	We now wish to establish that \(\liminf_{k \in K} \|F(\bar{x}) - u_k\|_2 = 0\).

	We first show that \(\{u_k\}_K\) has a bounded subsequence.
	If it were not the case, \(\liminf_{k \in K} \|F(\bar{x}) - u_k\|_2 = +\infty\).
	By convexity, \(h\) is bounded from below by an affine function, i.e., there exist \(a \in \R^p\) and \(b \in \R\) such that \(h(u) \geq a^\top u + b\) for all \(u \in \R^p\).
	Thus,
	\[
		\liminf_{k \in K} \tfrac{1}{2} \|F(\bar{x}) - u_k\|_2^2 + \mu_k (a^\top u_k + b) = \liminf_{k \in K} \tfrac{1}{2} \|F(\bar{x}) - u_k\|_2^2 + \mu_k a^\top u_k \leq 0,
	\]
	which can be written
	\[
		\liminf_{k \in K} \tfrac{1}{2} \|u_k\|^2 + (\mu_k a - F(\bar{x}))^\top u_k + \tfrac{1}{2} \|F(\bar{x})\|_2^2 \leq 0.
	\]
	But the above is impossible as the value of the left-hand side is \(+\infty\).
	Thus, reducing to a further subsequence if necessary,  we may assume without loss of generality that \(\{u_k\}_K\) is bounded.
	Consequently, \(\{F(\bar{x}) - u_k\}_K\) is bounded as well.

	Assume by contradiction that \(\liminf_{k \in K} \|F(\bar{x}) - u_k\|_2 \coloneqq \delta > 0\).
	Without loss of generality, we may assume that \(\{u_k\}_K \to \bar{u}\).
	We now obtain from~\eqref{eq:limit-inequality} that \(\tfrac{1}{2} \|F(\bar{x}) - \bar{u}\|_2^2 \leq 0\), which contradicts the assumption that \(\delta > 0\).
	The above establishes that \(\liminf_{k \in K} \|F(\bar{x}) - u_k\|_2 = 0\).
	Again, reducing to a further subsequence if necessary, we obtain that \(\bar{x}\) is asymptotically stationary for~\eqref{eq:constrained-convex-composite}.
\end{proof}

\subsection{Infeasibility detection}%
\label{sec:infeasibility-detection}

There are two ways in which~\eqref{eq:constrained-convex-composite} can be infeasible.
The first is that there is no \(x \in \R^n\) such that \(c(x) \in C\).
In order to detect such situation, we rest upon the subproblem solver since those constraints are kept explicitly in~\eqref{eq:partially-eliminated-subproblem}.
Thus, in this section, we assume that \(c(x) \in C\) is feasible.
The second way is that there is no \(x \in \R^n\) such that \(c(x) \in C\) and \(F(x) \in \dom h\).
In order to detect this second situation, we formulate the feasibility problem
\begin{equation}%
	\label{eq:feasibility-problem}
	\minimize{x \in \R^n}\quad
	\tfrac{1}{2} \dist(F(x) \mid \widebar{\dom h})^2
	\quad \st \quad
	c(x) \in C.
\end{equation}
In general,~\eqref{eq:feasibility-problem} is impractical as we may not know \(\dom h\) explicitly, let alone how to compute the distance to its closure.
However,~\eqref{eq:feasibility-problem} captures the concept of infeasibility for~\eqref{eq:constrained-convex-composite} because, as in the proof of \Cref{thm:convergence_AKKT}, if \(\{x_k\}\) has a feasible limit point, \(\{F(x_k) - u_k\} \to 0\) with \(u_k \in \dom h\) for all \(k\).
Thus, if~\eqref{eq:constrained-convex-composite} is feasible, we may hope that \(F(x_k) \in \widebar{\dom h}\) in the limit while the subproblem solver ensures that \(\dist(c(x_k) \mid C) \to 0\).

To identify cases where the optimal value of~\eqref{eq:feasibility-problem} is positive, we now argue that an appropriate proxy for~\eqref{eq:feasibility-problem} is
\begin{equation}%
	\label{eq:proxy-feasibility-problem}
	\minimize{x \in \R^n}\quad
	\mu h_\mu(F(x))
	\quad \st \quad
	c(x) \in C.
\end{equation}

The next result shows that, as \(\{\mu_k\} \to 0\), \(\mu_k h\) flattens to the zero function over the closure of its domain.

\begin{lemma}%
	\label{lem:epi-muh}
	For \(h: \R^n \to \R \cup \{+\infty\}\) proper, lsc and convex, and any positive nonincreasing sequence \(\{\mu_k\} \to 0\),
	\[
		\elim_{k \to \infty} \mu_k h = \chi(\cdot \mid \widebar{\dom h}).
	\]
\end{lemma}

\begin{proof}
	By \citep[Proposition~\(7.2\)]{rockafellar1998variational}, we must show that, for all \(x \in \R^n\),
	\begin{enumerate}
		\item%
		      \label{itm:elim-liminf}%
		      \(\liminf_{k \to \infty} \mu_k h(x_k) \geq \chi(x \mid \widebar{\dom h})\) for every sequence \(\{x_k\} \to x\), and
		\item%
		      \label{itm:elim-limsup}%
		      \(\limsup_{k \to \infty} \mu_k h(x_k) \leq \chi(x \mid \widebar{\dom h})\) for at least one sequence \(\{x_k\} \to x\).
	\end{enumerate}

	Let \(x \in \R^n\) and consider an arbitrary sequence \(\{x_k\} \to x\).
	Two situations can occur.
	Consider first the case where \(x \not \in \widebar{\dom h}\).
	For all sufficiently large \(k\), it must be that \(x_k \not \in \dom h\), and thus \(h(x_k) = \mu_k h(x_k) = +\infty\).
	Hence, \(\liminf_{k \to \infty} \mu_k h(x_k) = +\infty = \chi(x \mid \widebar{\dom h})\).

	Consider now the case where \(x \in \widebar{\dom h}\).
	Because \(h\) is proper and convex, it is bounded below by an affine function, i.e., there exist \(a \in \R^n\) and \(b \in \R\) such that \(h(x) \geq a^\top x + b\) for all \(x \in \R^n\).
	Thus, \(\mu_k h(x) \geq \mu_k a^\top x + \mu_k b\) for all \(k\), and \(\liminf_{k \to \infty} \mu_k h(x_k) \geq 0 = \chi(x \mid \widebar{\dom h})\).
	We have established condition~\ref{itm:elim-liminf}.

	We now establish condition~\ref{itm:elim-limsup}.
	Let \(x \in \R^n\).
	Three situations can occur.
	Consider first the case where \(x \not \in \widebar{\dom h}\), and let \(x_k \coloneqq x\) for all \(k\).
	Then, \(\limsup_{k \to \infty} \mu_k h(x_k) = +\infty = \chi(x \mid \widebar{\dom h})\).

	Consider next the case where \(x \in \dom h\), and let \(x_k \coloneqq x\) for all \(k\).
	Then, \(\limsup_{k \to \infty} \mu_k h(x_k) = 0 = \chi(x \mid \widebar{\dom h})\).

	Finally, consider the case where \(x \in \widebar{\dom h} \setminus \dom h\).
	There exists a sequence \(\{w_k\} \subseteq \dom h\) such that \(\{w_k\} \to x\).
	Let \(C_k \coloneqq \{x \in \R^n \mid |h(x)| \leq 1 / \sqrt{\mu_k}\} \subseteq \dom h\) for all \(k\).
	We have \(\dom h = \cup_k C_k\).
	Let \(q_0 \geq 0\) be the smallest integer such that \(w_{q_0} \in C_0\).
	Because \(\{\mu_k\}\) is nonincreasing, for each \(k \geq 1\), there exists a smallest integer \(q_k > q_{k-1}\) such that \(w_{q_k} \in C_k\).
	By construction, \(\{q_k\}\) is increasing, and therefore, \(\{w_{q_k}\} \to x\).
	Let \(x_k \coloneqq w_{q_k}\) for all \(k\).
	We obtain \(\mu_k |h(x_k)| \leq \sqrt{\mu_k} \to 0\), and hence, \(\limsup_k \mu_k h(x_k) = \lim_k \mu_k h(x_k) = 0 = \chi(x \mid \widebar{\dom h})\).
\end{proof}

The next result shows that epi-convergence is preserved under infimal convolution with the half squared Euclidean norm.

\begin{lemma}%
	\label{lem:inf-convolution}
	Let \(h_k\) and \(h: \R^n \to \R \cup \{+\infty\}\) be proper, lsc and convex, and let \(g \coloneqq \tfrac{1}{2}\|\cdot\|_2^2\).
	If \(\elim h_k = h\), then \(\elim (h_k \square g) = h \square g\).
\end{lemma}

\begin{proof}
	Under our assumptions, \citep[Theorem~\(11.34\)]{rockafellar1998variational} implies that \(\elim h_k^* = h^*\).
	Because \(g = g^*\), \citep[Exercise~\(7.8a\)]{rockafellar1998variational} yields \(\elim h_k^* + g^* = h^* + g^*\).
	Note that \(h_k^*\), \(h^*\) and \(g^*\) are proper, lsc and convex as well.
	The sum of Fenchel conjugates is related to the inf-convolution via \(h^* + g^* = (h \square g)^*\) \citep[Theorem~\(11.23a\)]{rockafellar1998variational}, and hence \(\elim (h_k \square g)^* = (h \square g)^*\).
	Thus, \citep[Theorem~\(11.34\)]{rockafellar1998variational} again implies \(\elim (h_k \square g) = h \square g\).
\end{proof}

We are now in position to establish the main result that links~\eqref{eq:feasibility-problem} and~\eqref{eq:proxy-feasibility-problem}.

\begin{theorem}%
	\label{thm:elim-moreau-envelope}
	For any positive nonincreasing sequence \(\{\mu_k\} \to 0\),
	\[
		\elim_{k \to \infty} \mu_k h_{\mu_k} = \tfrac{1}{2} \dist(\cdot \mid \widebar{\dom h})^2.
	\]
\end{theorem}

\begin{proof}
	Apply \Cref{lem:epi-muh} and \Cref{lem:inf-convolution} to \(h_k \coloneqq \mu_k h\) to obtain
	\[
		\elim_{k \to \infty} ((\mu_k h) \square g) =
		\chi(\cdot \mid \widebar{\dom h}) \square g =
		\tfrac{1}{2} \dist(\cdot \mid \widebar{\dom h})^2.
	\]
	The result follows from the observation that \((\mu_k h) \square g = \mu_k (h \square (\mu_k^{-1} g)) = \mu_k h_{\mu_k}\).
\end{proof}

The first-order optimality conditions for~\eqref{eq:proxy-feasibility-problem} are that there exist \(w \in C\) and \(z \in N_C(w)\) such that
\begin{equation}%
	\label{eq:KKT_proxy_feasibility}
	\nabla F(x) (F(x) - \prox{\mu h}(F(x))) + \nabla c(x) z = 0, \quad c(x) - w = 0.
\end{equation}
Then, in analogy with \eqref{eq:approxKKT}, it seems reasonable to declare local infeasibility in \Cref{alg:envelopt} if \(\{\mu_k\} \to 0\) and
\begin{subequations}%
	\label{eq:proxy-infeasibility-kkt}
	\begin{align}%
		\|\nabla F(x_k) (F(x_k) - u_k) + \nabla c(x_k) z\| \leq \epsilon^{\inf}
		,\quad
		\|c(x_k) - w\|                                     \leq{}& \epsilon^{\inf}
		\label{eq:proxy-infeasibility-kkt-1} \\
		\text{and}\quad
		\|F(x_k) - u_k\|                                   >{}& \epsilon,
		\label{eq:proxy-infeasibility-kkt-2}
	\end{align}%
\end{subequations}
for some \(w \in C\), \(z \in N_C(w)\), and prescribed tolerances \(\epsilon, \epsilon^{\inf} > 0\).
Criterion~\eqref{eq:proxy-infeasibility-kkt} is easily computable along the iterations and provides a well-defined termination condition, as we now demonstrate.

By~\eqref{eq:approxKKT_subproblem}, each iteration of \Cref{alg:envelopt} produces iterates \(x_k\), \(u_k = u(x_k; \hat{y}_k, \mu_k)\), \(y_k = y(x_k; \hat{y}_k, \mu_k) \in \partial h(u_k)\), \(w_k \in C\), and \(z_k \in N_C(w_k)\) such that
\[
	\|\nabla f(x_k) + \nabla F(x_k) y_k + \nabla c(x_k) z_k\| \leq \epsilon_k, \quad \|c(x_k) - w_k\| \leq \epsilon_k.
\]
The first condition can be written
\begin{equation}%
	\label{eq:proxy-infeasibility-mu}
	\|\mu_k \nabla f(x_k) + \nabla F(x_k) (\mu_k \hat{y}_k + F(x_k) - u_k) + \nabla c(x_k) (\mu_k z_k)\| \leq \mu_k \epsilon_k.
\end{equation}
Assume there exists an index set \(K \subseteq N\) such that \(\{\mu_k\}_K \to 0\).
If \(\nabla f(x_k)\) remains bounded, \(\{\mu_k \nabla f(x_k)\}_K \to 0\).
By \Cref{asm:envelopt-multipliers}, \(\{\mu_k \hat{y}_k\}_K \to 0\).
Because \(N_C(w_k)\) is a cone, \(\mu_k z_k \in N_C(w_k)\).
Therefore,~\eqref{eq:proxy-infeasibility-kkt} will be satisfied for sufficiently large \(k \in K\) if~\eqref{eq:constrained-convex-composite} is infeasible.

\subsection{Complexity}%
\label{sec:complexity}

Let \(\epsilon > 0\).
We first wish to bound the maximum number of iterations for \Cref{alg:envelopt} to return an \(\epsilon\)-stationary point of~\eqref{eq:constrained-convex-composite} in the event that \(\{\mu_k\}\) remains bounded away from zero.
We require the following assumption.

\begin{assumption}%
	\label{asm:complexity}
	A concrete implementation of \Cref{alg:envelopt} is such that there exists a positive sequence \(\{\eta_k\}\) and constants \(0 < \kappa_\mu < 1\), \(0 < \kappa_\eta < 1\) and \(0 < \kappa_\epsilon < 1\) satisfying the following conditions for all \(k\):
	\begin{enumerate}[label=(\roman{*})]
		\item \(\|F(x_k) - u_k\| > \eta_k\) \(\Longrightarrow\) \(\mu_{k+1} \leq \kappa_\mu \mu_k\);
		\item \(\|F(x_k) - u_k\| \leq \eta_k\) \(\Longrightarrow\) \(\mu_{k+1} = \mu_k\), \(\eta_{k+1} \leq \kappa_\eta \eta_k\), and \(\epsilon_{k+1} \leq \kappa_\epsilon \epsilon_k\).
	\end{enumerate}
\end{assumption}

The following result is inspired from \citep[Theorem~\(3.1\)]{birgin2020complexity}.

\begin{theorem}%
	\label{thm:complexity-mu-bounded}
    Let \Cref{asm:complexity} be satisfied, and let \(\epsilon_0 \geq \epsilon\).
    Without loss of generality, assume that \(\bar{\eta} \geq \epsilon\).
    Assume that each iteration of \Cref{alg:envelopt} succeeds, and that
	\begin{enumerate}[label=(\roman{*})]
		\item there exists \(\bar{\mu} > 0\) such that \(\mu_k \geq \bar{\mu}\) for all \(k\), and
		\item there exists \(\bar{\eta} \geq \epsilon\) such that \(\|F(x_k) - u_k\| \leq \bar{\eta}\) for all \(k\).
	\end{enumerate}
	\Cref{alg:envelopt} returns an \(\epsilon\)-stationary point of~\eqref{eq:constrained-convex-composite} in at most
	\begin{equation}%
		\label{eq:complexity-mu-bounded}
		\left\lceil \frac{ \log(\bar{\mu} / \mu_0) }{ \log(\kappa_\mu) } \right\rceil \, \max \left( \left\lceil \frac{ \log(\epsilon / \epsilon_0) }{ \log(\kappa_\epsilon) } \right\rceil, \, \left\lceil \frac{ \log(\epsilon / \bar{\eta}) }{ \log(\kappa_\eta) } \right\rceil \right)
	\end{equation}
	iterations.
\end{theorem}

\begin{proof}
    Because \(\mu_k\) never increases and is bounded away from zero, \Cref{asm:complexity} ensures that it eventually remains constant.
    When that occurs, \(\eta_k\) continues to decrease, so that \(F(x_k) - u_k \to 0\).
    In addition, there are at most \(\lceil \log(\bar{\mu} / \mu_0) / \log(\kappa_\mu) \rceil\) iterations \(k\) such that \(\|F(x_k) - u_k\| > \eta_k\).

	On iterations for which \(\|F(x_k) - u_k\| \leq \eta_k\), \(\epsilon_k\) decreases by a factor of at least \(\kappa_\epsilon\), and \(\eta_k\) decreases by a factor of at least \(\kappa_\eta\).
	Thus, there are at most \(\lceil \log(\epsilon / \epsilon_0) / \log(\kappa_\epsilon) \rceil\) such iterations until \(\epsilon_k \leq \epsilon\), and at most \(\lceil \log(\epsilon / \bar{\eta}) / \log(\kappa_\eta) \rceil\) such iterations until \(\eta_k \leq \epsilon\).

	In the worst case, primal feasibility \(\|F(x_k) - u_k\|\) returns to its largest value \(\bar{\eta}\) at each subproblem.
	Thus, the total number of iterations is at most~\eqref{eq:complexity-mu-bounded}.
\end{proof}

In \Cref{thm:complexity-mu-bounded}, we assumed without loss of generality that \(\bar{\eta} \geq \epsilon\).
If it were not the case, approximate primal feasibility would be satisfied throughout the iterations and the second term in the maximum in~\eqref{eq:complexity-mu-bounded} would simply not be present.
In all cases, the maximum number of iterations is \(O(\log(\epsilon))\).

\medskip

In general, though, \Cref{alg:envelopt} may stop at an iterate $x_k$ that appears to be infeasible and, at the same time, a local minimizer of the infeasibility problem \eqref{eq:proxy-feasibility-problem}.
The possibility that $\mu_k \to 0$ must be considered, since it necessarily takes place, for example, when \eqref{eq:constrained-convex-composite} is infeasible.
The following theorem establishes an upper bound on the number of iterations before Envelopt finds an approximate stationary point (in the sense of \eqref{eq:approxKKT}) or an infeasible point that is approximately stationary for the infeasibility problem \eqref{eq:proxy-feasibility-problem}, according to \eqref{eq:proxy-infeasibility-kkt}.

Let $\omega(k)\in\N$ denote the number of penalty updates up to the $k$-th iteration.
The update rules in \Cref{asm:complexity} give $\mu_k \leq \kappa_\mu^{\omega(k)}\mu_0$ and $\omega(k) \leq \ln(\mu_k/\mu_0)/\ln \kappa_\mu$.

\begin{theorem} \label{thm:complexity-mu-zero}
	Let $\epsilon > 0$ and $\epsilon^{\inf} \in (0,\epsilon]$ be given.
	Let \Cref{asm:complexity} be satisfied, and let \(\epsilon_0 \geq \epsilon\).
	Assume that
	\begin{enumerate}[label=(\roman{*})]
		\item \label{asm:complexity-mu-zero:etabar}%
		there exists \(\bar{\eta} \geq \epsilon\) such that \(\|F(x_k) - u_k\| \leq \bar{\eta}\) for all \(k\),
		\item \label{asm:complexity-mu-zero:c}%
		there exist \(c_f,c_F\in\R\) such that $\| \nabla f(x_k) \| \leq c_f$ and $\| \nabla F(x_k) \| \leq c_F$ for all $k$,
		\item \label{asm:complexity-mu-zero:Neps}%
		there exists \(N(\epsilon^{\inf}, \epsilon)\in\N\) such that \(\epsilon_k \leq \min\{\epsilon, \epsilon^{\inf}/4 \}\) for all \(k\geq N(\epsilon^{\inf}, \epsilon)\),
		\item \label{asm:complexity-mu-zero:muyhat}%
		there exist $\varrho_0>0$ and $\kappa_\varrho\in(0,1)$ such that \(\mu_k \|\hat{y}_k\| \leq \varrho_0 \kappa_\varrho^{\omega(k)}\) for all $k$.
	\end{enumerate}
	Then, \Cref{alg:envelopt} returns either an \(\epsilon\)-stationary point of~\eqref{eq:constrained-convex-composite}, according to \eqref{eq:approxKKT},
	or an \(\epsilon\)-infeasible $\epsilon^{\inf}$-stationary point of \eqref{eq:proxy-feasibility-problem}, according to \eqref{eq:proxy-infeasibility-kkt}, in at most
	\[
	\max\{ N(\epsilon^{\inf}, \epsilon), P(\epsilon) U(\epsilon^{\inf}) \}
	\]
	iterations, where
	$P(\epsilon)
	\coloneqq
	\ceil{\frac{\ln(\epsilon/\bar{\eta})}{\ln(\kappa_\eta)}}$
	and
	\[
		U(\epsilon^{\inf})
		\coloneqq
		\ceil{
		\max\left\{
			1,
			\frac{\ln(\epsilon^{\inf}/(4c_F\varrho_0))}{\ln\kappa_\varrho},
			\frac{\ln(\min\{ 1, \epsilon^{\inf}/(4c_f) \}/\mu_0)}{\ln\kappa_\mu}
		\right\}
		}.
	\]
\end{theorem}

Some comments on the assumptions are in order, particularly in relation to \cite[Theorem~\(3.5\)]{birgin2020complexity}.
Similarly to the boundedness requirement \ref{asm:complexity-mu-zero:etabar} on the constraint violation, the assumption \ref{asm:complexity-mu-zero:c} on the derivatives appears in \cite[Lemma 3.3]{birgin2020complexity} as a result of the boundedness of $\{x_k\}$ (which they assume directly).
In contrast, requirements \ref{asm:complexity-mu-zero:Neps} and \ref{asm:complexity-mu-zero:muyhat} are algorithmic: they pose restrictions on the sequences of hyperparameters selected by \Cref{alg:envelopt}, and not on the problem at hand.
While $N(\epsilon^{\inf}, \epsilon)$ in \ref{asm:complexity-mu-zero:Neps} is the same as in \cite[Thm~3.5]{birgin2020complexity}, the requirement \ref{asm:complexity-mu-zero:muyhat} is weaker than the (rigid) safeguarding condition adopted for the AL method in \cite{birgin2020complexity}.
Indeed, the BCL scheme of \citet{conn1991globally} is also covered; see \cite[Lemma 4.1]{conn1991globally}.

Before proving \cref{thm:complexity-mu-zero} we give an intermediate result concerning stationarity for the infeasibility problem \eqref{eq:proxy-feasibility-problem}.

\begin{lemma}\label{lem:complexity-mu-zero-infeas}
	Let $\epsilon > 0$ and $\epsilon^{\inf} \in (0,\epsilon]$ be given.
	Let \Cref{asm:complexity} be satisfied.
	Suppose \ref{asm:complexity-mu-zero:c}--\ref{asm:complexity-mu-zero:muyhat} from \cref{thm:complexity-mu-zero} are valid for the iterates generated by \cref{alg:envelopt}.
	Then, $k\geq N(\epsilon^{\inf},\epsilon)$ and $\omega(k) \geq U(\epsilon^{\inf})$ together imply \eqref{eq:proxy-infeasibility-kkt-1}.
\end{lemma}
\begin{proof}
	Fix $k\in\N$ and assume $k\geq N(\epsilon^{\inf},\epsilon)$ and $\omega(k) \geq U(\epsilon^{\inf})$ hold.
	Owing to the definition of $U(\epsilon^{\inf})$, the latter condition implies
	\begin{equation}\label{eq:complexity-thresholds}
		\varrho_0 \kappa_\varrho^{\omega(k)} \leq \epsilon^{\inf}/(4 c_F)
		\qquad\text{and}\qquad
		\mu_k \leq \min\{ 1, \epsilon^{\inf}/(4c_f) \} .
	\end{equation}
	For convenience, define the $C^1$ functions
	\begin{equation*}
		P_k(x)
		\coloneqq
		\mu_k h_{\mu_k}(F(x))
		\qquad\text{and}\qquad
		\widehat{P}_k(x)
		\coloneqq
		\mu_k h_{\mu_k}(F(x)+\mu_k\hat{y}_k)
	\end{equation*}
	to write more compactly the objective functions of \eqref{eq:proxy-feasibility-problem} and \eqref{eq:partially-eliminated-subproblem}, respectively.
	
	We proceed in three steps,
	(a) showing that $\|c(x_k)-w\| \leq \epsilon^{\inf}$ holds for some $w\in C$,
	(b) providing expressions for $\nabla P_k(x_k)$ and $\nabla \widehat{P}_k(x_k)$,
	and (c) establishing that $\| \nabla P_k(x_k) + \nabla c(x_k) z \| \leq \epsilon^{\inf}$ holds for some $z\in N_C(w)$.
	Together, these observations prove \eqref{eq:proxy-infeasibility-kkt-1}.
	
	(a)
	By \ref{asm:complexity-mu-zero:Neps}, $k\geq N(\epsilon^{\inf},\epsilon)$ implies $\epsilon_k \leq \epsilon^{\inf}$, hence subproblem \eqref{eq:partially-eliminated-subproblem} delivers $\|c(x_k)-w_k\| \leq \epsilon^{\inf}$ for some $w_k\in C$, by \eqref{eq:approxKKT_subproblem}.
	
	(b)
	Direct derivation of $P_k$ yields $\nabla P_k(x_k) = \nabla F(x_k) (F(x_k)-\prox{\mu_k h}(F(x_k)))$, using the differentiation rule for the Moreau envelope \eqref{eq:Moreau_envelope_deriv}.
	Similarly, using also the definition of $u_k = u(x_k;\hat{y}_k,\mu_k)$, it is $\nabla \widehat{P}_k(x_k) = \nabla F(x_k) (F(x_k) - u_k + \mu_k \hat{y}_k)$.
	
	(c)
	We find first a bound for $\| \nabla \widehat{P}_k(x_k) + \nabla c(x_k) z \|$, then for the difference $\| \nabla \widehat{P}_k(x_k) - \nabla P_k(x_k) \|$, and finally for $\| \nabla P_k(x_k) + \nabla c(x_k) z \|$ by combining the first two.
	It will be natural to consider the multiplier $z = \mu_k z_k \in N_C(w_k)$.
	From \eqref{eq:proxy-infeasibility-mu} and \eqref{eq:Moreau_envelope_deriv} we have that
	\begin{align*}
		\mu_k \epsilon_k
		\geq{}&
		\|\mu_k \nabla f(x_k) + \nabla F(x_k) (F(x_k) - u_k + \mu_k \hat{y}_k) + \nabla c(x_k) (\mu_k z_k) \|
		\\
		={}&
		\|\mu_k \nabla f(x_k) + \nabla \widehat{P}_k(x_k) + \nabla c(x_k) (\mu_k z_k) \|
	\end{align*}
	and therefore
	\begin{equation*}
		\begin{aligned}
			\| \nabla \widehat{P}_k(x_k) + \nabla c(x_k) (\mu_k z_k) \|
			\leq{}&
			\|\mu_k \nabla f(x_k) \| + \mu_k \epsilon_k
			\\
			\leq{}&
			\mu_k c_f + \mu_k \epsilon_k
			&&\text{[with $c_f$ from \ref{asm:complexity-mu-zero:c}]}
			\\
			\leq{}&
			\mu_k c_f + \mu_k \epsilon^{\inf}/4
			&&\text{[due to $k\geq N(\epsilon^{\inf},\epsilon)$]}
			\\
			\leq{}&
			\epsilon^{\inf}/2
			.
			&&\text{[by \eqref{eq:complexity-thresholds}]}
		\end{aligned}
	\end{equation*}
	Now we bound the difference between $\nabla\widehat{P}_k$ and $\nabla P_k$, obtaining
	\begin{equation*}
		\begin{aligned}
			\| \nabla{}& \widehat{P}_k(x_k) - \nabla P_k(x_k) \|
			\\
			={}&
			\left\|
			\nabla F(x_k) \left[ \prox{\mu_k h}(F(x_k)) + \mu_k\hat{y}_k - \prox{\mu_k h}(F(x_k)+\mu_k\hat{y}_k) \right]
			\right\|
			\\
			\leq{}&
			c_F \left\| \prox{\mu_k h}(F(x_k)) + \mu_k\hat{y}_k - \prox{\mu_k h}(F(x_k)+\mu_k\hat{y}_k) \right\|
			&&\text{[with $c_F$ from \ref{asm:complexity-mu-zero:c}]}
			\\
			\leq{}&
			2 c_F \mu_k \| \hat{y}_k \|
			\\
			\leq{}&
			2 c_F \varrho_0 \kappa_\varrho^{\omega(k)}
			\leq
			\epsilon^{\inf}/2,
			&&\text{[by \ref{asm:complexity-mu-zero:muyhat} and \eqref{eq:complexity-thresholds}]}
		\end{aligned}
	\end{equation*}
	where the second inequality follows from the nonexpansiveness of $\prox{\mu_k h}$ (unit Lipschitz constant by \cite[Thms 12.12, 12.17]{rockafellar1998variational}).
	Therefore, combining the bounds above we obtain
	\begin{equation*}
		\begin{aligned}
			\| \nabla P_k(x_k) + \nabla c(x_k) (\mu_k z_k) \|
			={}&
			\| \nabla P_k(x_k) - \nabla \widehat{P}_k(x_k) + \nabla \widehat{P}_k(x_k) + \nabla c(x_k) (\mu_k z_k) \|
			\\
			\leq{}&
			\| \nabla P_k(x_k) - \nabla \widehat{P}_k(x_k) \| + \| \nabla \widehat{P}_k(x_k) + \nabla c(x_k) (\mu_k z_k) \|
			\\
			\leq{}&
			\epsilon^{\inf}
			,
		\end{aligned}
	\end{equation*}
	concluding the proof.
\end{proof}

The main complexity result can now be established.

\begin{proof}[Proof of \cref{thm:complexity-mu-zero}]
	We proceed in two steps, always considering iterations $k \geq N(\epsilon^{\inf},\epsilon)$.
	First, we show that if \eqref{eq:approxKKT} is never satisfied, then \eqref{eq:proxy-infeasibility-kkt-2} holds and the number of iterations between two consecutive penalty updates is bounded above by $P(\epsilon)$.
	Second, we invoke \cref{lem:complexity-mu-zero-infeas}: if the number of penalty updates exceeds $U(\epsilon^{\inf})$, then \eqref{eq:proxy-infeasibility-kkt-1} holds.
	
	\emph{First step:}
	After the first $N(\epsilon^{\inf},\epsilon)$ iterations, $\epsilon_k\leq\epsilon$ by \ref{asm:complexity-mu-zero:Neps} guarantees that
	\[
	\| \nabla f(x_k) + \nabla F(x_k) y_k + \nabla c(x_k) z_k \| \leq \epsilon
	\quad\text{and}\quad
	\| c(x_k) - w_k \| \leq \epsilon
	\]
	are satisfied for some suitable $w_k\in C$ and $z_k\in N_C(w_k)$.
	Therefore, \eqref{eq:proxy-infeasibility-kkt-2} must hold if \eqref{eq:approxKKT} fails.

	Because $\|F(x_k)-u_k\| \leq \bar{\eta}$ for all $k$ by assumption \ref{asm:complexity-mu-zero:etabar}, the conditions in \Cref{asm:complexity} indicate that it takes at most $P(\epsilon) \coloneqq \ceil{\ln(\epsilon/\bar{\eta}) / \ln(\kappa_\eta)}$ consecutive iterations without penalty updates to reach $\|F(x_k)-u_k\| \leq \epsilon$.
	This means that, for $k\geq N(\epsilon^{\inf}, \epsilon)$, penalty updates are at most $P(\epsilon)$ iterations apart as long as \eqref{eq:approxKKT} fails.	
	
	\emph{Second step:}
	We now invoke \cref{lem:complexity-mu-zero-infeas}, which guarantees that condition \eqref{eq:proxy-infeasibility-kkt-1} is satisfied after at most $N(\epsilon^{\inf},\epsilon)$ iterations and $U(\epsilon^{\inf})$ penalty updates.
	
	Combining these observations, the total number of iterations to satisfy either \eqref{eq:approxKKT} or \eqref{eq:proxy-infeasibility-kkt} is at most $N(\epsilon^{\inf},\epsilon) + P(\epsilon) U(\epsilon^{\inf})$.
\end{proof}

\section{Extensions and applications}%
\label{sec:extensions-applications}

\subsection{Standard augmented Lagrangian method}

A simple observation that is worth pointing out is that the standard augmented Lagrangian method is a special case of \Cref{alg:envelopt}.
Suppose indeed that \(h = 0\) in~\eqref{eq:constrained-convex-composite}, and that we reformulate it equivalently as
\[
	\minimize{x \in \R^n}\quad
	f(x) + \chi(c(x) \mid C).
\]
The latter has the form~\eqref{eq:constrained-convex-composite} with \(h = \chi(\cdot \mid C)\), \(F = c\) and no explicit constraints.
Each Envelopt subproblem involves the Moreau envelope of the indicator of \(C\), i.e.,
\[
	\chi_\mu = \chi(\cdot \mid C) \square (\tfrac{1}{2} \mu^{-1} \|\cdot\|_2^2) = \tfrac{1}{2} \mu^{-1} \dist(\cdot \mid C)^2.
\]
Hence, each Envelopt subproblem has the familiar form
\[
	\minimize{x \in \R^n}\quad
    f(x) + \tfrac{1}{2} \mu^{-1} \dist(c(x) + \mu \hat{y} \mid C)^2.
\]

\subsection{NCL extension}\label{sec:ncl}

As a strategy to solve constrained problems that do not satisfy a constraint qualification, \citet{ma2018stabilized} proposed the NCL formulation of the augmented Lagrangian method.
The main idea is to relax the hard constraints \(c(x) \in C\) by introducing \emph{residual} variables \(r\) such that \(c(x) - r \in C\), and to penalize the constraint \(r = 0\) via an augmented Lagrangian.
The same idea applied to~\eqref{eq:constrained-convex-composite} yields augmented Lagrangian subproblems of the form
\[
	\minimize{x \in \R^n, r \in \R^m}\quad
    f(x) + h(F(x)) + \tfrac{1}{2} \rho^{-1} \|r + \rho \hat{z}\|_2^2
	\quad \st \quad
	c(x) - r \in C,
\]
where \(\rho > 0\) is a penalty parameter and \(\hat{z} \in \R^m\) is a vector of Lagrange multiplier estimates.
Advantages of the formulation above are that the subproblems are always feasible, always satisfy LICQ, and that solvers are better able to exploit sparsity of \(\nabla c\) than if \(c(x) \in C\) were penalized explicitly in the objective via an augmented Lagrangian.
We refer the interested reader to \citep{ma2018stabilized,ma2021julia} for more details on the NCL formulation and its implementation.

The smoothing strategy of \(h(F(x))\) described in \Cref{sec:methodology} can be applied to the above subproblem to yield subproblems of the form
\[
	\minimize{x \in \R^n, r \in \R^m}\quad
    f(x) + h_\mu(F(x) + \mu \hat{y}) + \tfrac{1}{2} \rho^{-1} \|r + \rho \hat{z}\|_2^2
	\quad \st \quad
	c(x) - r \in C.
\]
One may choose \(\rho = \mu\) in each subproblem, or to manage the two penalty parameters separately.

The combined Envelopt/NCL formulation may be obtained by applying Envelopt to the formulation
\[
	\minimize{x \in \R^n}\quad
	f(x) + h(F(x)) + \chi(c(x) \mid C),
\]
without explicit constraints and the combined proper, lsc and convex nonsmooth term \((u, v) \mapsto h(u) + \chi(v \mid C)\).
In that sense, Envelopt generalizes NCL\@.

\subsection{An exact penalty method}\label{sec:exact_penalty_method}

Consider the smooth constrained problem
\begin{equation}%
	\label{eq:equality-constrained-problem}
	\minimize{x \in \R^n}\quad
	f(x)
	\quad \st \quad
	F(x) = 0 ,\quad \underline{x} \leq x \leq \widebar{x}.
\end{equation}
By way of introducing slack variables, any smooth constrained problem with equality and inequality constraints can be cast in the form~\eqref{eq:equality-constrained-problem}.
The exact penalty method \citep{pietrzykowski1969exact} consists in solving a sequence of subproblems of the form
\begin{equation}%
	\label{eq:exact-penalty-subproblem}
	\minimize{x \in \R^n}\quad
	f(x) + \tau \|F(x)\|
	\quad \st \quad
	\underline{x} \leq x \leq \widebar{x},
\end{equation}
where \(\|\cdot\|\) is a norm and \(\tau > 0\) is a penalty parameter.
The latter subproblem has the general form~\eqref{eq:constrained-convex-composite} with \(h = \|\cdot\|\), \(c = \identity\) and \(C = [\underline{x}, \, \widebar{x}]\).
The proximal operator of \(h\), and hence its Moreau envelope, is known and is related to the projection into the unit ball in the dual norm.
Most often, \(h\) is the \(\ell_1\) or Euclidean norm.
\Cref{alg:envelopt} provides a way to address the exact penalty subproblem~\eqref{eq:exact-penalty-subproblem} based on a smoothing of \(h\) via its Moreau envelope.
\Cref{alg:exact-penalty} summarizes the resulting algorithm to solve~\eqref{eq:equality-constrained-problem} via exact penalty.

\begin{algorithm}[ht]%
	\caption[caption]{%
		\label{alg:exact-penalty}
		Exact penalty method based on \Cref{alg:envelopt}.
	}
	\begin{algorithmic}[1]%
		\State Choose a stopping tolerance \(\epsilon > 0\), \(0 < \kappa_\epsilon < 1\), and \(\kappa_\tau > 1\).
		\State Choose \(\epsilon_0 \geq \epsilon\), and \(\tau_0 > 0\).
		\For{\(k = 0, 1, \ldots\)}
		\State Find an \(\epsilon_k\)-stationary point \(x_k\) of~\eqref{eq:exact-penalty-subproblem} with \(\tau = \tau_k\) using \Cref{alg:envelopt}.
		\State If \(x_k\) is \(\epsilon\)-stationary for~\eqref{eq:equality-constrained-problem}, return \(x_k\).
		\State Set \(\tau_{k+1} \coloneqq \kappa_\tau \tau_k\) and \(\epsilon_{k+1} = \kappa_\epsilon \epsilon_k\).
		\EndFor
	\end{algorithmic}
\end{algorithm}

Each subproblem~\eqref{eq:exact-penalty-subproblem} is bound constrained and can be solved with any appropriate solver for bound-constrained optimization.
Had bound constraints not been present in~\eqref{eq:equality-constrained-problem}, the subproblems would have been unconstrained.
Similarly, any constraint that was part of \(F(x) = 0\) could have been kept explicit provided a general constrained solver is available.
In this sense, the approach above is more general than that described by \citet{diouane-gollier-orban-2026}, who consider only equality-constrained problems and must appeal to a solver for nonsmooth problems to solve the subproblems.
When \(\|\cdot\|\) is a polyhedral norm,~\eqref{eq:exact-penalty-subproblem} can be reformulated as a smooth problem, though one with general inequality constraints.
The approach above yields smooth subproblems with either simple bounds or no constraints at all.

\subsection{Proximal operator of a composite sum}\label{sec:prox_composite_sum}

Let \(g: \R^n \to \R \cup \{+\infty\}\) be proper, lsc, and have a known proximal operator.
To evaluate the proximal operator of \((h\circ F) + g\), with parameter $\lambda>0$, we must solve a problem of the form
\begin{equation}\label{eq:prox_composite_sum}
	\minimize{x \in \R^n}\quad
    \tfrac{1}{2} \lambda^{-1} \|x - u\|_2^2 + h(F(x)) + g(x),
\end{equation}
whose solutions are generally not known explicitly.
However, \Cref{alg:envelopt} can be used to approximate one.
Indeed, each subproblem~\eqref{eq:partially-eliminated-subproblem} takes the form
\[
	\minimize{x \in \R^n}\quad
    \tfrac{1}{2} \lambda^{-1} \|x - u\|_2^2 + h_\mu(F(x) + \mu \hat{y}) + g(x),
\]
and can be solved with, e.g., the proximal-gradient method.

\section{Implementation and numerical experiments}%
\label{sec:implementation-numerical-experiments}

\subsection{Realizations of Envelopt}%
\label{sec:concrete_implementations_envelopt}

Two concrete implementations of \Cref{alg:envelopt} inspired directly by popular augmented Lagrangian methods in the literature satisfy \Cref{asm:envelopt-multipliers}.
The first is the algorithm of \citet{conn1991globally}, which we restate as \Cref{alg:envelopt-bcl}.

\begin{proposition}
	\Cref{alg:envelopt-bcl} satisfies \Cref{asm:envelopt-multipliers}.
\end{proposition}

\begin{proof}
	To establish the first condition in \Cref{asm:envelopt-multipliers}, consider $\epsilon=0$.
    If there exists $\bar{\mu}>0$ such that $\mu_k \geq \bar{\mu}$, $\mu_k$ can only be updated finitely many times, and the updates of Line~\ref{step:bcl-update-alpha-valid} are performed infinitely many times for all sufficiently large \(k\).
	Therefore, since $\alpha_{k+1} \in (0,1)$ for each $k$, we obtain $\eta_k \to 0$, and hence, $F(x_k)-u_k\to 0$.
	The second condition in \Cref{asm:envelopt-multipliers} was established in \citep[Lemma 4.1]{conn1991globally}.
\end{proof}

\begin{algorithm}[ht]%
	\caption[caption]{%
		\label{alg:envelopt-bcl}
		Envelopt algorithm modeled after \citep{conn1991globally}.
	}
	\begin{algorithmic}[1]%
		\State Choose stopping tolerance \(\epsilon > 0\).
		\State Choose \(0 < \kappa_\mu < 1\), \(0 < \gamma < 1\), \(\alpha_\epsilon > 0\), \(\alpha_\eta > 0\), \(\beta_\epsilon > 0\), and \(\beta_\eta > 0\).
		\State Choose \(\epsilon_1 \geq \epsilon\), \(\eta_1 \geq \epsilon\), \(\mu_0 > 0\) and \(\hat{y}_0 \in \R^p\).
		\For{\(k = 0, 1, \ldots\)}
		\State Find an \(\epsilon_k\)-stationary point \(x_k\) of~\eqref{eq:partially-eliminated-subproblem}.
		\State Compute \(u_k = u(x_k; \hat{y}_k, \mu_k)\) and \(y_k = y(x_k; \hat{y}_k, \mu_k)\).
		\State%
		If \(x_k\) is \(\epsilon\)-stationary for~\eqref{eq:constrained-convex-composite}, return \(x_k\).
		\If{\(\|F(x_k) - u_k\| \leq \max(\epsilon, \eta_k)\)}
		\State Set \(\mu_{k+1} \coloneqq \mu_k\) and \(\hat{y}_{k+1} \coloneqq y_k\).
		\State Set \(\alpha_{k+1} \coloneqq \min(\gamma, \mu_{k+1})\), \(\epsilon_{k+1} \coloneqq \epsilon_k \alpha_{k+1}^{\beta_\epsilon}\), and \(\eta_{k+1} \coloneqq \eta_k \alpha_{k+1}^{\beta_\eta}\).\label{step:bcl-update-alpha-valid}
		\Else
		\State Set \(\mu_{k+1} \coloneqq \kappa_\mu \mu_k\) and \(\hat{y}_{k+1} \coloneqq \hat{y}_k\).
		\State Set \(\alpha_{k+1} \coloneqq \min(\gamma, \mu_{k+1})\), \(\epsilon_{k+1} \coloneqq \epsilon_1 \alpha_{k+1}^{\alpha_\epsilon}\), and \(\eta_{k+1} \coloneqq \eta_1 \alpha_{k+1}^{\alpha_\eta}\).
		\EndIf
		\EndFor
	\end{algorithmic}
\end{algorithm}

Another widespread scheme that complies with \cref{asm:envelopt-multipliers} is that of \citet{birgin2014practical}, see also \citep{andreani2008augmented,demarchi2023constrained,demarchi2025augmented}.
Following this approach the multiplier estimates $\hat{y}_{k+1}$ are always updated, but \emph{safeguarded} by projecting $y_k$ onto a suitable compact set, so that the product $\mu_k \hat{y}_k$ vanishes as $\mu_k$ approaches zero.

\subsection{Implementation details}%
\label{sec:implementation-details}

In our implementation, several subproblem solvers are available depending on \(m\), the number and types of hard constraints.
If \(m = 0\), subproblems are unconstrained and are solved by way of a traditional trust-region method \citep[Chapter~\(6\)]{conn-gould-toint-2000} in which steps are computed by the truncated conjugate gradient (CG) method \citep{steihaug-1983}.
Our implementation uses the \texttt{trunk} solver \citep{migot_jsosolvers_jl_unconstrained_and_2026}.

If \(m > 0\) but the hard constraints are simple bounds, subproblems are bound constrained.
We solve them with our own implementation of the trust-region method for bound-constrained problems of \citet{tron}, named \texttt{tron}, where steps are computed using a combination of projected-gradient steps and the truncated CG method.
Whereas the original code of \citep{tron} requires explicit exact second-derivatives and uses a limited-memory factorization as preconditioner for CG, our implementation, also found in \citep{migot_jsosolvers_jl_unconstrained_and_2026}, does not require exact second derivatives, but does not use a preconditioner.
In particular, the latter can be used with quasi-Newton approximations.

When \(m > 0\) and \(c\) represents constraints that are more general than simple bounds, we appeal to a general constrained solver and currently support three of them: the filter interior-point method implemented in Ipopt \citep{waechter-biegler-2006,orban2026nlpmodelsipopt} and MadNLP \citep{shin2024accelerating}, and the various solvers available as part of the KNITRO library \citep{byrd-nocedal-waltz-2006,migot_nlpmodelsknitro_jl_a_thin_2026}.
Currently, we use the \emph{direct} interior-point method implemented in KNITRO \citep{waltz-morales-nocedal-orban-2006}.
The main reason for relying on interior-point methods is that our implementation warm starts them as the outer iterations progress in a manner similar to what is done in the NCL solver \citep{ma2018stabilized,ma2021julia,montoison2025madncl}.
Naturally, it is also possible to use Ipopt, MadNLP or KNITRO when the subproblems are bound constrained or unconstrained.

An important feature of~\eqref{eq:partially-eliminated-subproblem} is that \(h_\mu\) is not twice differentiable in general.
Thus, we rely on quasi-Newton approximations.
Ipopt, MadNLP and KNITRO all support quasi-Newton approximations, but their management is hard-coded in those libraries in such a way that is crucial for the efficiency of the step computation---see, e.g., \citep{waltz-morales-nocedal-orban-2006}.
However, \texttt{trunk} and \texttt{tron} let users supply their own approximations.
With those two solvers, we implemented structured quasi-Newton approximations in the sense that we supply the Hessian approximation
\[
	\nabla^2 f(x_{k,j}) + B_{k,j},
\]
where \(B_{k,j} = B_{k,j}^\top\) is a limited-memory BFGS or SR1 approximation constructed from pairs
\[
	s_{k,j} \coloneqq x_{k,j+1} - x_{k,j},
	\qquad
	y_{k,j} \coloneqq \nabla h_\mu(F(x_{k,j}) + \mu_k \hat{y}_k) - \nabla h_\mu(F(x_{k,j-1}) + \mu_k \hat{y}_k),
\]
where \(k\) denotes an iteration of \Cref{alg:envelopt} and \(j\) one of the subproblem solver.

\subsection{Numerical experiments}%
\label{sec:num_experiments}

The broad applicability of Envelopt is demonstrated in this section with illustrative and real-world problems.
We also compare the performance of our Envelopt\footnote{\texttt{Envelopt v0.1.0}, \https{github.com/JuliaSmoothOptimizers/Envelopt.jl}.} implementation with relevant solvers, such as
MadNLP\footnote{\texttt{MadNLP v0.9.2}, \https{github.com/MadNLP/MadNLP.jl}.} \citep{shin2024accelerating},
Ipopt \citep{waechter-biegler-2006,orban2026nlpmodelsipopt},
SCS \citep{odonoghue2016conic},
Clarabel \citep{goulart2026clarabel},
NCL \citep{ma2018stabilized,ma2021julia},
and
ALPS \citep{demarchi2023constrained}.

\subsubsection{Degenerate SDP}

\newcommand{\PSDcone}[1]{\mathcal{S}_+^{#1}}

We consider an illustrative linear conic problem by \citet[Example~4]{ramana1997exact}, which reads
\begin{equation}%
	\label{eq:problem_sdp}
	\minimize{x\in\R^2}{}\quad -x_2 \qquad
	\st \quad
	F(x)\coloneqq\begin{bmatrix}
		1-x_2 & 0    & 0    \\
		0     & -x_1 & -x_2 \\
		0     & -x_2 & 0
	\end{bmatrix} \in \PSDcone{3}
\end{equation}
where $\PSDcone{3}$ denotes the cone of $3\times 3$ positive semidefinite matrices.
Due to lack of constraint qualifications,~\eqref{eq:problem_sdp} and its dual problem exhibit a finite and positive duality gap, which makes the problem degenerate and hard to solve.
Indeed, SCS \citep{odonoghue2016conic}, an operator-splitting conic solver, and Clarabel \citep{goulart2026clarabel}, an interior-point solver tailored to conic problems, struggle with~\eqref{eq:problem_sdp}.

For all tolerances $\epsilon\in\{10^{-3},10^{-4},10^{-5},10^{-6}\}$,
SCS\footnote{\texttt{SCS v2.6.3}, \https{github.com/jump-dev/SCS.jl}.} hits the maximum number of iterations ($10^5$ by default) and returns with status \texttt{solved (inaccurate - reached max\_iters)}, while Clarabel\footnote{\texttt{Clarabel v0.11.1}, \https{github.com/oxfordcontrol/Clarabel.jl}.} stops after 38 iterations because of \texttt{insufficient progress}.

Problem~\eqref{eq:problem_sdp} fits~\eqref{eq:constrained-convex-composite} with the linear cost in $f$, $h$ as the indicator of $\PSDcone{3}$, and $F$ the affine matrix-valued operator appearing in~\eqref{eq:problem_sdp}.
Then, subproblems~\eqref{eq:partially-eliminated-subproblem} take the unconstrained form
\[
    \minimize{x\in\R^2}{}\quad -x_2 + \tfrac{1}{2} \mu^{-1} \dist(F(x)+\mu\hat{y} \mid \PSDcone{3})^2 .
\]
Envelopt is able to find $\epsilon$-stationary points for~\eqref{eq:problem_sdp}, even with $\epsilon=10^{-6}$, as reported in \cref{tab:results_sdp}.
MadNLP performs better than Ipopt and KNITRO in this case, tackling the subproblems down to tolerance $\epsilon=10^{-6}$ and with fewer total inner iterations.

\begin{table}[tbh]
	\centering%
	\caption{Solving~\eqref{eq:problem_sdp} with Envelopt. Comparison for different subsolvers and tolerances $\epsilon$ in terms of number of iterations, and total number of subsolver iterations. A dash indicates solver failure.}%
	\label{tab:results_sdp}%
	\begin{tabular}{c|ccc|ccc|ccc}
		\hline
		subsolver            & \multicolumn{3}{c|}{MadNLP} & \multicolumn{3}{c|}{Ipopt} & \multicolumn{3}{c}{KNITRO}                                                                         \\
		tolerance $\epsilon$ & $10^{-4}$                   & $10^{-5}$                  & $10^{-6}$                  & $10^{-4}$ & $10^{-5}$ & $10^{-6}$ & $10^{-4}$ & $10^{-5}$ & $10^{-6}$ \\
		\hline
		iterations  			& 9	& 7 & 7 							& 6	& 19	& - 						& 6 & 4 & - \\
		subsolver iter.			& 61 & 72 & 69 							& 61 & 125	&  -						& 111 & 75 & - \\
		\hline
	\end{tabular}
\end{table}

\subsubsection{Simple MPCC}

We showcase the benefits of integrating Envelopt and NCL, as discussed in \cref{sec:ncl}, by considering the following two-dimensional problem with a complementarity constraint, inspired by \citet{scholtes2001convergence}:
\begin{align}\label{eq:mpec_reg}
	\minimize{x\in\R^2}\quad{} & \|x-1\|^2 + \|x\|_1 &
	\st\quad{}                 &
	x \geq 0
	,\quad
	x_1 x_2 \leq 0	.
\end{align}
This example also allows us to highlight the flexibility offered by Envelopt in the problem formulation, particularly the encoding of constraints as soft or hard.

The objective of~\eqref{eq:mpec_reg} is convex, and the feasible set is connected but nonconvex.
Moreover, standard constraint qualifications fail at the origin, which is feasible.
There are two global minimizers, $(0.5,0)$ and $(0,0.5)$, with objective value $1.75$.
The origin is a local maximizer, with objective value $2$.

Problem~\eqref{eq:mpec_reg} fits the template~\eqref{eq:constrained-convex-composite} with $f(x)\coloneqq \|x-1\|^2$, $F\coloneqq \identity$, $h\coloneqq \|\cdot\|_1$, nonnegativity bounds, and a nonlinear inequality constraint.
We run Envelopt and its NCL extension for~\eqref{eq:mpec_reg}, along with Ipopt, MadNLP and NCL \citep{orban_ncl_jl_a_julia, ma2021julia}
for the equivalent reformulation of~\eqref{eq:mpec_reg} as the smooth nonlinear program
\begin{align}\label{eq:mpec_nlp}
	\minimize{x\in\R^2}\quad{} & x_1^2+x_2^2-x_1-x_2+2 &
	\st\quad{}                 &
	x \geq 0
	,\quad
	x_1 x_2 \leq 0										.
\end{align}
Starting from 100 random initial points, each solver terminated with a successful status and approximately at $(0,0.5)$, $(0.5,0)$, or $(0,0)$ in at least 96 cases.
A summary of the numerical performance is reported in \cref{tab:results_mpec}.
Only the NCL variants and Envelopt with Ipopt consistently find the global minimum; the other solvers returned the local maximizer $(0,0)$ in 24--27\% of cases.

\begin{table}[tbh]
	\centering%
	\caption{Solving a simple MPCC from 100 random initial points with Envelopt, NCL-Envelopt, and NLP solvers. Comparison for different problem formulations and subsolvers in terms of success rate in finding the global minimum and total number of subsolver iterations. Tolerance $\epsilon=10^{-5}$.}%
	\label{tab:results_mpec}%
	\begin{tabular}{c|c|cc|cc}
		\hline
		problem 								& solver (subsolver)	& global 	& local 	& \multicolumn{2}{c}{subsolver iterations} \\
		    									&           			& min rate 	& min rate	&  (median) & (max) \\
		\hline
		\multirow{4}{*}{\eqref{eq:mpec_reg}} 	& Envelopt (MadNLP) 	& 71\%		& 27\%		& 62.5		& 154 \\
												& Envelopt (Ipopt) 		& 97\%		& 0\%		& 25 		& 106 \\
												& NCL-Envelopt (MadNLP) & 100\%		& 0\%		& 154		& 679 \\
												& NCL-Envelopt (Ipopt) 	& 100\%		& 0\% 		& 54 		& 104 \\
		\hline
		\multirow{3}{*}{\eqref{eq:mpec_nlp}} 	& MadNLP 				& 69\% 		& 27\%	 	& 24 		& 31 \\
												& Ipopt 				& 76\%		& 24\%	 	& 26 		& 48 \\
												& NCL 					& 100\%		& 0\%		& 13 		& 81 \\
		\hline
		\multirow{2}{*}{\eqref{eq:mpec_bnd}} 	& Envelopt (MadNLP) 	& 73\% 		& 27\% 		& 142.5 	& 194 \\
												& Envelopt (Ipopt) 		& 73\% 		& 27\%	 	& 57.5 		& 98 \\
												& Envelopt (\texttt{tron}) 		& 54\%		& 46\% 		& 97.5 		& 334 \\
		\hline
		\multirow{2}{*}{\eqref{eq:mpec_unc}} 	& Envelopt (MadNLP) 	& 100\%		& 0\%		& 79		& 121 \\
												& Envelopt (Ipopt) 		& 100\% 	& 0\% 		& 67.5  	& 91  \\
												& Envelopt (\texttt{tron}) 		& 100\% 	& 0\% 		& 3553.5	& 4112  \\
												& Envelopt (\texttt{trunk}) 		& 74\% 		& 7\% 		& 537 		& 5503 \\
		\hline
	\end{tabular}
\end{table}

We use~\eqref{eq:mpec_reg} to inspect the effect of formulating constraints as hard or soft.
As discussed in \Cref{sec:introduction}, Envelopt offers the option of dealing with constraints directly in the subsolver,
in contrast with other schemes that relax all constraints.
Let us rewrite~\eqref{eq:mpec_reg} as the equivalent problem
\begin{align}\label{eq:mpec_bnd}
	\minimize{x\in\R^2}\quad{} & \|x-1\|^2 + \|x\|_1 + \chi(x_1x_2 \mid \R_-) &
	\st\quad{}                 &
	x \geq 0																	.
\end{align}
With this formulation, the nonlinear inequality constraint $x_1 x_2 \leq 0$ is not passed to the subsolver but instead relaxed and treated with an augmented Lagrangian penalty in the Envelopt iterations; the corresponding subproblems~\eqref{eq:partially-eliminated-subproblem} have only bound constraints.
Another equivalent formulation of~\eqref{eq:mpec_reg} is
\begin{equation}\label{eq:mpec_unc}
	\minimize{x\in\R^2}\quad{} \|x-1\|^2 + \|x\|_1 + \chi(x \mid \R_+^2) + \chi(x_1x_2 \mid \R_-) ,
\end{equation}
which leads to unconstrained subproblems.
Note that there is no NCL counterpart for~\eqref{eq:mpec_bnd} and~\eqref{eq:mpec_unc}, since they lack nonlinear explicit constraints.

The performance of Envelopt on formulations~\eqref{eq:mpec_bnd} and~\eqref{eq:mpec_unc} is summarized in \cref{tab:results_mpec}, for comparison with~\eqref{eq:mpec_reg}.
For each Envelopt subsolver, the formulation~\eqref{eq:mpec_reg} with hard constraints requires significantly fewer iterations than~\eqref{eq:mpec_bnd} and~\eqref{eq:mpec_unc}.
A tentative explanation is that simple constraints are better handled directly by the subsolver than by the Moreau envelope relaxation and smoothing in the outer iterations.

\subsubsection{Low-rank matrix completion}

Consider $N$ points $p_1, \ldots, p_N \in \R^\ell$ and $P \coloneqq \begin{bmatrix} p_1 & p_2 & \cdots & p_N \end{bmatrix} \in \R^{\ell \times N}$.
Let $\Delta\in\R^{N\times N}$ denote the associated Euclidean distance matrix, i.e., $\Delta_{i,j} \coloneqq \|p_i-p_j\|^2$ for all $i$, $j = 1, \ldots, N$.
We wish to recover information on $P$ based on partial knowledge of $\Delta$.
In particular, we assume that $\Omega \subset \{1, \ldots, N\}^2$ is a set of pairs such that only the entries $\Delta_{i,j}$, $(i, j) \in\Omega$, of $\Delta$ are known.

We lift the problem by introducing $B \coloneqq P^\top P$ whose rank is, by construction, less than or equal to $\ell$.
We seek a symmetric matrix $B \in \R^{N\times N}$ that satisfies the distance constraints associated with the observations.
Among these admissible matrices, those with minimum rank are preferred.
We formulate the matrix completion task as \citet[\S\(4.5\)]{demarchi2023constrained}, for comparison with ALPS\@:
\begin{align}\label{eq:problem_matcompletion}
	\minimize{B\in\R^{N\times N}}\quad{} & \|B\|_*                                        &  &                                        \\
	\st\quad{}                           & B_{i,i}-B_{i,j}-B_{j,i}+B_{j,j} = \Delta_{i,j} &  & \forall (i,j)\in\Omega 		\nonumber     \\
	                                     & B_{i,j} = B_{j,i}                              &  & \forall i,j\in\{1,\ldots,N\}	\nonumber
\end{align}
where the nuclear norm $\|\cdot\|_* \coloneqq \sum_i \sigma_i(\cdot)$ encourages $B$ to have sparse singular values $\sigma_i(B)$, hence low rank \citep{recht2010guaranteed}.

The problem has the form~\eqref{eq:constrained-convex-composite} with \(x = \mathrm{vec}(B)\), \(f = 0\), \(h(x) = \|\mathrm{mat}(x)\|_*\), \(F(x) = x\), and \(c(x)\) represents the explicit constraints above, where we denoted \(\mathrm{vec}\) the operation that stacks the columns of a matrix into a vector, and \(\mathrm{mat}\) its inverse.
The problem has \(n=N^2\) variables and \(m=|\Omega| + N(N-1)/2\) constraints.

We generated random instances as in \citep[\S 4.5]{demarchi2023constrained}, with $N=10$, $\ell=5$, and $|\Omega|=\lfloor N(N + 1)/6 \rfloor$ observations, leading to problems~\eqref{eq:constrained-convex-composite} with 100 variables and 63 (linear equality) constraints.
\cref{tab:results_matrix_completion} summarizes the results obtained by invoking Envelopt on 100 random instances, each with a random initialization, for different tolerance levels.

\begin{table}[tbh]
	\centering%
	\caption{Solving 100 random instances of~\eqref{eq:problem_matcompletion} with Envelopt. Comparison for different subsolvers and tolerances $\epsilon$ in terms of success rate, number of iterations, and total number of subsolver iterations.}%
	\label{tab:results_matrix_completion}%
	\begin{tabular}{cc|cccc|cccc}
		\hline
		\multicolumn{2}{c|}{Envelopt subsolver}   & \multicolumn{4}{c|}{MadNLP} & \multicolumn{4}{c}{Ipopt}                                                                                \\
		\multicolumn{2}{c|}{tolerance $\epsilon$} & $10^{-3}$                   & $10^{-4}$                 & $10^{-5}$ & $10^{-6}$ & $10^{-3}$ & $10^{-4}$ & $10^{-5}$ & $10^{-6}$        \\
		\hline
		\multicolumn{2}{c|}{success rate}         & 97\%                        & 97\%                      & 93\%      & 90\%      & 89\%      & 87\%      & 76\%      & 60\%             \\
		\multirow{2}{*}{iterations}               & (median)                    & 6                         & 7         & 7         & 8         & 5         & 7         & 8         & 9    \\
		                                          & (max)                       & 9                         & 12        & 14        & 16        & 8         & 8         & 9         & 10   \\
		subsolver                                 & (median)                    & 14                        & 14        & 14        & 13        & 113       & 167       & 217.5     & 266  \\
		iterations                                & (max)                       & 444                       & 801       & 1068      & 854       & 592       & 1451      & 1019      & 1538 \\
		\hline
	\end{tabular}
\end{table}

Across all tolerances $\epsilon$, the failures of Envelopt arise from the subsolver returning prematurely with an error message, always for subproblems~\eqref{eq:partially-eliminated-subproblem} with $\mu > 10^{-3}$.
Rather than structural limitations of the Envelopt approach, this indicates the need for a better integration and tuning of subsolvers in our implementation.
Conversely, Envelopt is significantly more iteration-efficient than ALPS, an augmented Lagrangian scheme without the partial minimization step, which required several thousands of projected-gradient steps, see \citep[Fig. 5]{demarchi2023constrained}.

\subsubsection{Exact penalty method}

An interesting application of Envelopt is the solution of \emph{exact penalty subproblems}, as mentioned in \cref{sec:exact_penalty_method}.
Envelopt can address the subproblem~\eqref{eq:exact-penalty-subproblem} for different choices of penalty function $h$.
We illustrate this by invoking \cref{alg:exact-penalty} on two problems from the CUTEst benchmark set \citep{gould2015cutest}.
The small problem \texttt{zangwil3} has $n=3$ variables and $m=3$ equality constraints.
The medium-sized problem \texttt{hs118} has $n=32$ and $m=17$ in the form~\eqref{eq:equality-constrained-problem}.

The results reported in \cref{tab:results_exact_penalty_hs18} indicate that Envelopt can effectively address the subproblems arising from the exact penalty method.
Comparing subsolvers and penalty norms, it appears that Ipopt and $h\coloneqq \|\cdot\|_1$ perform slightly better than other configurations.

\begin{table}[tbh]
	\centering%
	\caption{Solving CUTEst problems with \cref{alg:exact-penalty} and Envelopt for~\eqref{eq:exact-penalty-subproblem}. Comparison for different choices of penalty norms and subsolvers, in terms of iterations for \cref{alg:exact-penalty} on~\eqref{eq:equality-constrained-problem}, for Envelopt on~\eqref{eq:exact-penalty-subproblem}, and the total number of subsolver iterations for each call to Envelopt.}%
	\label{tab:results_exact_penalty_hs18}%
	\begin{tabular}{ccc|ccc|ccc}
		                        &                                     &                                       & \multicolumn{3}{c|}{problem \texttt{hs118}} & \multicolumn{3}{c}{problem \texttt{zangwil3}}                                                             \\
		\hline
		Envelopt                & \multicolumn{2}{c|}{number of}      & \multicolumn{3}{c|}{penalty norm $h$} & \multicolumn{3}{c}{penalty norm $h$}                                                                                                                    \\
		subsolver               & \multicolumn{2}{c|}{iterations}     & $\|\cdot\|_1$                         & $\|\cdot\|_2$                               & $\|\cdot\|_\infty$                            & $\|\cdot\|_1$ & $\|\cdot\|_2$ & $\|\cdot\|_\infty$        \\
		\hline
		\multirow{5}{*}{MadNLP} & \multicolumn{2}{c|}{penalty method} & 6                                     & 12                                          & 7                                             & 4             & 3             & 5                         \\
		                        & \multirow{2}{*}{Envelopt}           & (median)                              & 6                                           & 9.5                                           & 6             & 2             & 3                  & 3    \\
		                        &                                     & (max)                                 & 8                                           & 13                                            & 10            & 3             & 4                  & 4    \\
		                        & \multirow{2}{*}{subsolver}          & (median)                              & 428.5                                       & 1435                                          & 439           & 4             & 6                  & 7    \\
		                        &                                     & (max)                                 & 696                                         & 4134                                          & 2201          & 19            & 12                 & 17   \\
		\hline
		\multirow{5}{*}{Ipopt}  & \multicolumn{2}{c|}{penalty method} & 7                                     & 7                                           & 6                                             & 4             & 6             & 2                         \\
		                        & \multirow{2}{*}{Envelopt}           & (median)                              & 6                                           & 6                                             & 6             & 3             & 3                  & 3.5  \\
		                        &                                     & (max)                                 & 8                                           & 14                                            & 9             & 4             & 3                  & 4    \\
		                        & \multirow{2}{*}{subsolver}          & (median)                              & 356                                         & 359                                           & 338           & 8.5           & 7                  & 14.5 \\
		                        &                                     & (max)                                 & 503                                         & 1550                                          & 719           & 16            & 13                 & 19   \\
		\hline
	\end{tabular}
\end{table}

\subsubsection{Structured composite sum}
We test Envelopt on the use case mentioned in \Cref{sec:prox_composite_sum}.
Given a convex regularizer $h$, a matrix $A\in\R^{p\times n}$, and a prox-friendly regularizer $g$, the proximal operator of $\varphi \coloneqq g+h\circ A$ corresponds to solving the problem \eqref{eq:prox_composite_sum}, with $F(x) \coloneqq Ax$, for some prescribed $\xi\in\R^n$ and $\lambda>0$.
Since the quadratic $f(x)\coloneqq \frac{1}{2\lambda}\|x-\xi\|^2$ is convex smooth and $h$ is convex prox-friendly, problem \eqref{eq:prox_composite_sum} can be tackled with the three-term splitting schemes V\~u--Condat \cite{condat2013primal,vu2013splitting} and AFBA \cite{latafat2017asymmetric} when also $g$ is convex prox-friendly.
In contrast, Envelopt can address \eqref{eq:prox_composite_sum} also with nonconvex $g$, nonconvex $f$, and $h$ composed with a nonlinear operator.

We choose $h \in \left\{ \|\cdot\|_1, \|\cdot\|_2 \right\}$, $g \in \left\{ \|\cdot\|_1, \|\cdot\|_{1/2}^{1/2}, \|\cdot\|_0 \right\}$, set $n=p=10$, sample $A$, $\xi$ from a normal distribution, and take values $\lambda \in \{10^{-3},\ldots,10^3\}$ equidistant on a logarithmic scale.
For the Envelopt subproblems we use the proximal-gradient methods R2DH \cite[\S 7]{diouane2026proximal} and NMPG \cite{demarchi2023monotony} through their implementation \texttt{R2DH} and \texttt{NMPG} in \texttt{RegularizedOptimization.jl} \cite{gollier2026regularized}.
While both methods use a backtracking strategy for globalization,
R2DH requires at least a fraction of Cauchy decrease and NMPG adopts an Armijo-type condition.
They also employ different merit values for nonmonotone globalization, with max and average formulas, respectively. 
Envelopt with these subsolvers is compared against the primal-dual methods implemented as \texttt{VuCondat} and \texttt{AFBA} in \texttt{ProximalAlgorithms.jl} \cite{stella2025proximalalgorithmsjl}.
All solvers are given a tolerance $\epsilon = 10^{-5}$, maximum $10^5$ iterations, and initialized with a random $x_0\in\R^n$ (and zero dual).
Although the solvers have different stopping conditions, our objective here is to show that Envelopt can solve problems that the other two solvers cannot.

Envelopt always found an $\epsilon$-stationary point, with both subsolvers.
V\~u--Condat and AFBA solved all instances with $g\coloneqq \|\cdot\|_1$, and otherwise failed on about 40\% of instances, regardless of $h$.
In the convex setting, where V\~u--Condat and AFBA have convergence guarantees, Envelopt requires fewer iterations for small values $\lambda$, and slightly more for large values.
The statistics reported in \cref{tab:results_f_composite_sum} indicate that Envelopt behaves consistently across different regularizers $g$ and $h$, regardless of the subsolver.

\begin{table}[tbh]
	\centering%
	\caption{Computing the proximal operator of a composite sum \eqref{eq:prox_composite_sum}. Comparison for different choices of functions $h$ and $g$, in terms of total number of iterations. While Envelopt solved all instances, V\~{u}--Condat and AFBA failed on about 40\% of instances with $g$ nonconvex; their statistics are omitted in this case.}%
	\label{tab:results_f_composite_sum}%
	\begin{tabular}{c|c|cc|cc|cc}
		& & \multicolumn{2}{c|}{$g=\|\cdot\|_1$} & \multicolumn{2}{c|}{$g=\|\cdot\|_{1/2}^{1/2}$} & \multicolumn{2}{c}{$g=\|\cdot\|_0$} \\
		\hline
		& iterations & median & max & median & max & median & max \\
		\hline
		\multirow{4}{*}{\rotatebox{90}{$h=\|\cdot\|_1$}} 		& V\~{u}--Condat 	& 12 & 34403		& - & - 		& - & - \\
																& AFBA 				& 59 & 1219 		& - & -			& - & - \\
																& Envelopt (R2DH) 	& 18 & 1363			& 13 & 953		& 13 & 789 \\
																& Envelopt (NMPG) 	& 10 & 362 			& 7 & 864		& 5.5 & 625 \\
		\hline
		\multirow{4}{*}{\rotatebox{90}{$h=\|\cdot\|_2$}} 		& V\~{u}--Condat 	& 12 & 4538 		& - & - 		& - & - \\
																& AFBA 				& 59 & 2817 		& - & - 		& - & - \\
																& Envelopt (R2DH) 	& 16 & 140			& 9 & 127 		& 7.5 & 138 \\
																& Envelopt (NMPG) 	& 7 & 134 			& 4 & 97 		& 5 & 93 \\
		\hline
	\end{tabular}
\end{table}

\section{Discussion and perspectives}%
\label{sec:discussion-perspectives}

Our implementation of \Cref{alg:envelopt} could clearly be improved in a number of ways.
One such way is to use generalized Hessians of the Moreau envelope as discussed by \citet{dhingra-khong-jovanovic-2022}.
Whether generalized Hessians yield better performance in practice than quasi-Newton approximations remains to be seen.
A second way is for cases where the proximal operator of \(h\) is not known in analytic form, but can be approximated by an iterative procedure that can be interrupted early.
Inexact evaluations of the proximal operator imply inexact evaluations of the Moreau envelope and of its gradient, but recent research indicates that such inexact evaluations can be incorporated into proximal algorithms while preserving their convergence and worst-case complexity properties \citep{demarchi2022proximal,allaire-digabel-orban-2025}.

An obvious open question is whether any of the above generalizes to nonconvex \(h\).
In view of \citep[Proposition~\(13.37\)]{rockafellar1998variational}, if \(h\) is \emph{prox regular}, its proximal operator is locally single-valued and Lipschitz continuous for sufficiently small values of \(\mu\).
Moreover, its Moreau envelope remains differentiable with gradient~\eqref{eq:Moreau_envelope_deriv}.
Thus, borrowing ideas from \citep{demarchi2024implicit}, a generalization may be possible though it appears difficult in general to determine values of \(\mu\) for which those properties occur.

\bibliographystyle{abbrvnat}
\bibliography{abbrv,envelopt}

\clearpage
\tableofcontents

\end{document}